\documentclass[a4paper]{amsart} % ,9pt
\usepackage{amsmath}
\usepackage{amssymb}%This gives us symbols like \square
\usepackage{amsthm}%This gives us Definition style of theorem.
\usepackage{mathrsfs}%This gives us \mathscr
\usepackage{graphicx}
\usepackage{mathtools}
\usepackage[colorlinks,pdfdisplaydoctitle,linkcolor=blue,citecolor=red]{hyperref}
\usepackage{caption}
\usepackage{subcaption}
\usepackage{url}
\usepackage{epstopdf}
\usepackage{pgf,tikz}
\usepackage{todonotes}
\usepackage{bbm}
\usepackage{extarrows}
\usetikzlibrary{arrows}
\usepackage{algorithm}
 \usepackage{algorithmicx}
 \usepackage{algpseudocode}
\usepackage{amssymb,amsthm}   
\usepackage{graphicx}        
\usepackage{amsmath}         
\usepackage{caption}
\usepackage{subcaption}
\usepackage{dsfont}
\usepackage{listings}
\usepackage{bbm}
\usepackage{color}
\usepackage{ dsfont }
\usepackage{enumerate}
 \usepackage{relsize}
\usepackage{xfrac}

\usepackage{changes}
\newcommand{\stkout}[1]{\ifmmode\text{\sout{\ensuremath{#1}}}\else\sout{#1}\fi}
\setdeletedmarkup{\stkout{#1}}

\renewcommand{\Delta}{\triangle}

\newcommand{\bbP}{\mathbb P}

\newcommand{\tr}{\text{tr}}

\usepackage{pifont}% http://ctan.org/pkg/pifont
\usepackage{hhline}

\definecolor{darkblue}{rgb}{0,0,0.7}

\definecolor{darkgreen}{rgb}{0.01,0.75,0.24}

\def \Ee[#1]{\mathcal{E}^{\text{{#1}}}}
\def\R{\mathbf{R}}

\def\pa[#1,#2]{\frac{\partial {#1}}{\partial {#2}} }

\def\idom[#1,#2,#3]{\int_{#1}\hspace{1pt} {#2} \hspace{1pt} \text{d}{#3}}
\def\res[#1,#2]{\left.{#1}\right|_{#2}}

\def\var[#1,#2]{\langle \delta \mathcal{E}^{\text{{#1}}}({#2}),v\rangle}
\def\vars[#1,#2,#3]{\langle \delta^2\mathcal{E}^{\text{{#1}}}({#2})v,{#3}\rangle}
\def\vard[#1,#2,#3,#4]{\langle \delta\mathcal{E}^{\text{{#1}}}({#2})-\delta\mathcal{E}^{\text{{#3}}}({#4}),v\rangle}

\def\P{\mathbb{P}}

\def\E{\mathbb{E}}

\def\N{\mathbb{N}}

\newcommand{\Tr}{\mathrm{Tr}}
\newcommand{\calU}{\mathcal{U}}
\newcommand{\calD}{\mathcal{D}}
\newcommand{\calA}{\mathcal{A}}

\newcommand{\calL}{\mathcal{L}}

\newcommand{\Fhat}{\widehat{F}}
\newcommand{\lambdahat}{\hat{\lambda}}

\DeclareMathOperator*{\argmin}{arg\,min}

\newcommand{\be}{\begin{equation}}
\newcommand{\en}{\end{equation}}
\newcommand{\ben}{\begin{equation*}}
\newcommand{\enn}{\end{equation*}}
\newcommand{\bea}{\begin{aligned}}
\newcommand{\ena}{\end{aligned}}

\def\ba#1\ena{\begin{align}#1\end{align}}
\def\ban#1\enan{\begin{align*}#1\end{align*}}

\theoremstyle{plain}
\newtheorem{thm}{Theorem}[section]

\newtheorem{lem}[thm]{Lemma}

\newtheorem{prop}[thm]{Proposition}
\newtheorem{assumption}[thm]{Assumption}

\definecolor{mypink1}{rgb}{0.858, 0.188, 0.478}

\newtheorem{remark}[thm]{Remark}

\numberwithin{equation}{section}

\mathtoolsset{showonlyrefs=true}

\begin{document}

%\title[Adaptive Tikhonov regularization schemes for EKI]{Adaptive Tikhonov regularization schemes for  ensemble Kalman inversion}
\title[Consistency of bilevel data-driven learning in inverse problems]{Consistency analysis of bilevel data-driven learning in inverse problems}

\author[N. K. Chada] {Neil K. Chada}
\address{Applied Mathematics and Computational Science Program, King Abdullah University of Science and Technology, Thuwal, 23955, KSA}
\email{neilchada123@gmail.com}

\author[C. Schillings] {Claudia Schillings}
\address{Mannheim School of Computer Science and Mathematics, University of Mannheim, 68131 Mannheim, Germany}
\email{c.schillings@uni-mannheim.de}

\author[X. T. Tong] {Xin T. Tong}
\address{Department of Mathematics, National University of Singapore, 119077, Singapore}
\email{mattxin@nus.edu.sg}

\author[S. Weissmann] {Simon Weissmann}
\address{Interdisciplinary Center for Scientific Computing, University of Heidelberg, 69120 Heidelberg, Germany}
\email{simon.weissmann@uni-heidelberg.de}

\begin{abstract}

{
One fundamental problem when solving  inverse problems is how to find regularization parameters. 
This article considers solving this problem using data-driven bilevel optimization{, i.e. we consider the adaptive learning of the regularization parameter from data by means of optimization.}
This approach can be interpreted as solving an empirical risk minimization problem, and {we analyze its performance in the large data sample size limit for general} nonlinear problems. {We demonstrate how to implement our framework on linear inverse problems, where we can further show the inverse accuracy does not depend on the ambient space dimension.}
To reduce the associated computational cost, online numerical schemes are derived using  the stochastic gradient descent method. {We prove} convergence of these numerical schemes under suitable assumptions on the forward problem. 
Numerical experiments are presented illustrating the theoretical results and demonstrating the applicability and efficiency of the proposed approaches for various linear and nonlinear inverse problems, including Darcy flow, the eikonal equation, and an image denoising example.}
\end{abstract}

\maketitle
\textbf{Keywords}: {bilevel optimization, statistical consistency, inverse problems, \\ stochastic gradient descent, Tikhonov regularization}

\section{Introduction}
\label{sec:intro}
Data-driven {modeling  seeks to  improve model accuracy and predictability by exploiting informations from existing data}. 
{It has lead to a wide range of successes} in deep learning, reinforcement learning, natural language processing and others \cite{FTH01,GBC16,SB14}.  
{This article is interested in its applications when solving inverse problems. Mathematically speaking, when solving an inverse problems, we try to recover a} $u \in \mathcal{U}$ from a perturbed data $y \in \mathcal{Y}$ where their relationship is given as
\begin{equation}
\label{eq:inverse_p}
y = \mathcal{G}(u) + \eta.
\end{equation}
{In \eqref{eq:inverse_p},} {$\eta$ denotes an additive observational noise} and $\mathcal{G}:\mathcal{U} \rightarrow \mathcal{Y}$ is the mapping from the parameter space $\mathcal U$ to the observation space $\mathcal Y$. Here, $\mathcal U$ and $\mathcal Y$ denote possibly infinite dimensional Banach spaces. Solutions to inverse problems have been well-studied through the use of variational and optimization methods which are well-documented in the following texts \cite{EHN96,AT87}. 

{Regularization is an important aspect of the numerical treatment of inverse problems.  It helps} overcoming the ill-posedness problem {in theory and the overfitting phenomenon in practice. It can also be interpreted as a form of a-priori knowledge in  the Bayesian approach \cite{KS04,AMS10}.} {To implement regularization on} \eqref{eq:inverse_p},  we estimate the unknown parameter $u$ by minimizing a regularized loss function, i.e. we consider
\begin{equation}
\label{eq:fun}
{u_{\lambda}} := \argmin_{u \in \mathcal{U}}\ \mathcal L_{\mathcal Y}(\mathcal{G}(u),y) + \mathcal S_\lambda(u), \quad \lambda\in\R_+,
\end{equation}
where $\mathcal L_{\mathcal Y}:\mathcal Y\times\mathcal Y\to\R_+$ is some metric in $\mathcal Y$ and $S_\lambda:\mathcal U\to\R_+$ is a regularization function with regularization parameter $\lambda>0$. A common choice is Tikhonov regularization \cite{TLY98} which can be included in \eqref{eq:fun} through the penalty term $S_\lambda(u) = \frac\lambda2\|u\|_{\mathcal U}^2$. The choice of norm $\|\,\cdot\,\|_{\mathcal U}$ often models prior information on the unknown parameter. Other common forms include $L_1$ and total variation regularization, which are particularly useful for imaging purposes \cite{BB18,EHN96,LP13}.

In \eqref{eq:fun}, the parameter $\lambda$ balances the influence of the data and the a-priori knowledge via the regularization. While expert knowledge can often provide a rough range of $\lambda$, the exact value, i.e. the $\lambda$ leading to the best estimation of the unknown parameter $u$, is often difficult to determine. However, the parameter $\lambda$ strongly influences the accuracy of the estimate and has to be properly chosen. Bilevel optimization is one way to resolve this issue \cite{CMS07,SD02,SMD18,AT87}. It seeks to learn the regularization parameter in a variational manner, and it can be viewed as a data--driven regularization \cite{AMOS19}. To formulate this approach, we view unknown parameter $U\in\mathcal U$ and the data $Y\in\mathcal Y$ in the model \eqref{eq:inverse_p} as a jointly distributed random variable with distribution $\mu_{(U,Y)}$. To find the best possible regularization parameter of the model \eqref{eq:inverse_p}, the bilevel minimization seeks to solve
\begin{equation}\label{eq:bi-level_opt}
\begin{alignedat}{2}
\lambda_\ast &= \underset{\lambda>0}{\argmin}\ F(\lambda),\quad F(\lambda)=\E_{\mu_{(U,Y)}}[\mathcal{L}_{\mathcal{U}}(u_\lambda(Y),U)],\quad &\text{(upper level)}\\
u_{\lambda}(Y)&:=\underset{u\in\mathcal U}{\argmin}\ \mathcal{L}_{\mathcal{Y}}(\mathcal G(u),Y) + \mathcal{S}_{\lambda}(u), &\text{(lower level)}
\end{alignedat}
\end{equation}
where $\mathcal L_{\mathcal U}:\mathcal U\times\mathcal U\to\R_+$ is some metric in the parameter space $\mathcal U$. The upper level problem seeks to minimize the distance between the unknown parameter $U$ and the regularized solution corresponding to its data $Y$, which is computed through $u_\lambda(Y)$ in the lower level problem. To solve this (stochastic) bilevel optimization problem, we assume that we have access to training data, given through samples of $(U_i,Y_i)\sim\mu_{(U,Y)}$, and the function $F$ in \eqref{eq:bi-level_opt} can be approximated by its empirical Monte--Carlo approximation. The area of bilevel optimization has been applied to various methodologies for inverse problems.  
To motivate this we provide various examples of the application of bilevel optimization, in the setting describe by \eqref{eq:bi-level_opt}, to inverse problems and an overview of recent literature.

\subsection{Motivating Examples}

\subsubsection{Example 1 - PDE-constrained inverse problems}
\label{ssec:pde} 
We first consider a inverse problem \eqref{eq:inverse_p} with the lower level problem formulated by a partial differential equation (PDE):
\begin{equation}\label{eq:IP_PDE_constr}
\begin{split}
&\argmin_{u \in \mathcal{U}}\  \mathcal L_{\mathcal{Y}}(\mathcal{O}(p), y) + \mathcal{S}_{\lambda}(u), \\  & \vspace{5em} \textrm{s.t.} \ M(u,p)=0,
\end{split}
\end{equation}
 where $u\in\mathcal U$ denotes the unknown parameter and $p\in\mathcal V$ is the state.
The function $M:\mathcal U\times \mathcal V\to\mathcal W$ describes  an underlying ODE or PDE model. The operator $\mathcal O:\mathcal V\to\R^K$ denotes the observation operator which maps the state $p$ to finite dimensional observations.  {The Darcy's flow problem is one such example. In particular, $u$ describes a subsurface structure, $p$ is the corresponding pressure field, $M$ describes the Darcy's law, and $\mathcal O$ evaluate $p$ at different locations.} 
  
 In order to formulate the corresponding bilevel problem \eqref{eq:bi-level_opt}, we assume that the forward model $M(u,p)=0$ is well-posed, which means that for each $u\in\mathcal U$ there exists a unique $p\in\mathcal V$ such that $M(u,p)=0\in\mathcal W$. Hence, using the solution operator $G:\mathcal U\to\mathcal V$ s.t.~$M(u,G(u))=0$, we can formulate the reduced problem of \eqref{eq:IP_PDE_constr} by
\begin{equation}\label{eq:IP_PDE_constr_reduced}
\argmin_{u \in \mathcal{U}}\ \mathcal L_{\mathcal{Y}}(\mathcal G(u),y)+ \mathcal{S}_{\lambda}(u),
\end{equation}
 where we have defined $\mathcal G=\mathcal O\circ G$. Hence, given a training data set $(u^{(j)},y^{(j)})$ we can also formulate the empirical bilevel minimization problem
 \begin{equation}\label{eq:PDE_bileve}
\begin{split}
\widehat\lambda_{n} &:= \argmin_{\lambda>0} \frac{1}{n}\sum^{n}_{j=1}\|u_{\lambda}(y^{(j)}) - u^{(j)}\|^2_2,\\
u_\lambda(y^{(j)}) &:= \argmin_{u \in \mathcal{U}}\ \mathcal L_{\mathcal Y}(\mathcal{G}(u),y^{(j)}) + S_\lambda(u).
\end{split}
\end{equation}

In terms of applications, many inverse problems arising in PDEs \cite{YYB02} are concerned with the recovery of an unknown which is heterogeneous. As a result it is very natural to model the unknown as a Gaussian random fields. Such models include Darcy flow, the Navier--Stokes model \cite{KBJ14} and electrical impedance tomography \cite{LB02,KS04}. Physical constraints such as boundary, or initial conditions are required for modeling purposes.

Holler et al. \cite{HKB18} consider bilevel optimization for inverse problems in the setting of \eqref{eq:IP_PDE_constr}. They provide theory which suggests existence of solutions and formulate their problem as an optimal control problem. This is connected with the work of Kaltenbacher \cite{KKV14,BK16} who  provided a modified approach known as ``all-at-once" inversion. These works have also been used in the context of deconvolution \cite{CC011,CEN14,SCCKT18}.%, which have considered
 
\subsubsection{Example 2 - Image \& signal processing} \label{ssec:image} Bilevel optimization is a popular solution choice for image processing problems \cite{BMO08,KP13}. In these problems, one is interested in optimizing over an underlying image and particular areas/segments of that image. A common example of this includes image denoising which is to remove noise from an image.  Another example is image deblurring where the image is commonly given as a convolution with a linear kernel $\mathcal{A}$, i.e.
$$
y = \mathcal{A}\ast u+\eta,
$$
{where $\ast$ denotes the convolution of $\mathcal{A}$ and $u$, commonly expressed as
$$
\mathcal{A}\ast u{(x)}  = \int_{\R^d}\mathcal{A}(x-\tau)u(\tau)d\tau.
$$}
 This inverse problem is also known as deconvolution. The setting of \eqref{eq:bi-level_opt} is common for deconvolution, where their loss functions are given as
\begin{equation}\label{eq:imaging_bilevel}
\begin{split}
\widehat\lambda_{n} &:= \argmin_{\lambda>0} \frac{1}{n}\sum^{n}_{j=1}\|u_{\lambda}(y^{(j)}) - u^{(j)}\|^2_2,\\
u_\lambda(y^{(j)}) &:= \argmin_{u \in \mathcal{U}}\ \mathcal L_{\mathcal Y}(\mathcal{A}*u,y^{(j)}) + \lambda\|Lu\|^2_2.
\end{split}
\end{equation}
In \eqref{eq:imaging_bilevel}, $L$ is a regularization matrix, and the upper level problem is taken as the minimization of the empirical loss function given by a training data set $(u^{(j)},y^{(j)})$. Commonly $\lambda$ is taken to be either a weighted function between $\mathcal L_{\mathcal Y}$ and the penalty term, or it can be viewed as the noise within a system. Common choices of $L$ traditionally are $L=I$ or a first or second order operator, which can depend on the unknown or image of interest. Further detail on the choice of $L$  and $\mathcal{A}$ are discussed in \cite{BMO08}.

The work of De los Reyes, Sch\"onlieb \cite{CCDSV16,DDT20,DS13,DSV17} and coauthors considered the application of bilevel optimization to denoising and deblurring, where non-smooth regularization is used such as total variation and Bregman regularization. The latter forms of regularization are useful in imaging as they preserve non-smooth features, such as edges and straight lines. {Recent developments of these techniques using Bayesian methodologies for imaging can be found in \cite{VDPD20,DDVP20}}.

\subsection{Our contributions}
In this article, we investigate two different {approaches to solve bilevel optimization} and their performance on  inverse problems. 
Firstly we formulate the offline approach of bilevel optimization as an empirical risk minimization (ERM) problem. 
Analyzing the performance of the ERM solution is difficult, since the loss function is random and non-convex, so numerical solutions often can only find local minimums.
We build  a theoretical framework under these general settings. In particular, it provides convergence rate of the ERM solution when sample size goes to infinity.  This framework is applied to linear inverse problems to understand the performance of bilevel optimization approach. {Moreover, our results depend only on the effective dimension, but not the ambient space dimension. This is an important aspect in inverse problems since the underlying space can be of infinite dimension.}

Secondly, we discuss how to implement stochastic gradient descent (SGD) methods on bilevel optimization. SGD is a popular optimization  tool for empirical risk minimization because of its straightfoward implementation and efficiency. The low computational costs are particularly appealing in the bilevel context as  finding the lower-level solution is already time consuming. Besides exact SGD, we also consider SGD with central difference approximation. This can be useful for problems with complicated forward observation models. A general consistency analysis framework is formulated for both exact SGD and approximated SGD. We demonstrate how to apply this framework to linear inverse problems.

Various numerical results are presented highlighting and verifying the theory acquired. Our results are firstly presented on  various partial differential equations both linear and nonlinear which include Darcy flow and the eikonal equation, as motivated through Example 1 in subsection \ref{ssec:pde}. {We also test our theory on an image denoising example which is discussed through Example 2 in subsection \ref{ssec:image}. In particular, we demonstrate numerically the statistical consistency result which includes the rate of convergence and we show the learned parameter $\lambda$ within each inverse problem experiment outperforms that with a fixed $\lambda$.}

{We emphasize that with our findings and results in this work, our focus is not on developing new methodology for bilevel learning. Instead our focus is on building a statistical understanding of bilevel learning through the notion of statistical consistency and convergence of numerical schemes.}

\subsection{Organization}
The structure of this paper is given as follows. In Section \ref{sec:stoc} we present the bilevel optimization problem of interest in a stochastic framework, and present a statistical consistency result of the adaptive parameter. We then extend this result to the linear inverse setting with Tikhonov regularization. This will lead onto Section \ref{sec:reg} where we discuss the stochastic gradient descent method and its application in our bilevel optimization problem. We provide various assumptions required where we then show in the linear setting that our parameter converges in $L_2$ to the minimizer. In Section \ref{sec:num} we test our theory on various numerical models which include both linear and nonlinear models such as Darcy flow and the eikonal equation. This will also include an image denoising example.  Finally, we conclude our findings in Section \ref{sec:conc}. The appendix will contain the proofs for results in Section \ref{sec:stoc} and Section \ref{sec:reg}.

\section{Regularization parameter offline recovery}
\label{sec:stoc}
In this section we  discuss how to use offline bilevel optimization to recover regularization parameters. We also show the solution is statistically consistent under suitable conditions. 

\subsection{Offline bilevel optimization}
 Regularization parameter learning by bilevel optimization views the unknown parameter $U$ and the data $Y$ as a jointly distributed random variable with distribution $\mu_{(U,Y)}$, see e.g. \cite{AMOS19} for more details. Recall the bilevel optimization problem is given by
\begin{equation*}
\begin{alignedat}{2}
\lambda_\ast &= \underset{\lambda\in \Lambda}{\argmin}\ F(\lambda),\quad F(\lambda)=\E_{\mu_{(U,Y)}}[\mathcal{L}_{\mathcal{U}}(u_\lambda(Y),U)],\quad &\text{(upper level)}\\
u_{\lambda}(Y)&:=\underset{u\in\mathcal U}{\argmin}\ \Psi(\lambda,u,Y),\quad \Psi(\lambda,u,y):=\mathcal{L}_{\mathcal{Y}}(\mathcal G(u),y) + \mathcal{S}_{\lambda}(u), &\text{(lower level)}
\end{alignedat}
\end{equation*}
where $\mathcal{L}_{\mathcal{U}}$ denotes a discrepancy function in the parameter space $\mathcal{U}:=\R^d$ and $\mathcal L_{\mathcal{Y}}$ denotes a discrepancy function in the observation space $\mathcal{Y}:=\R^K$. $S_\lambda(U)$ represents the regularization with  parameter $\lambda\in \Lambda$. Here, $\Lambda$  represents the range of regularization parameters which often come from physical constraints. For simplicity, we assume all the functions here are continuous and integrable, and so are their first and second order derivatives with respect to $\lambda$.  

In general, we do not know the exact distribution $\mu$ in the upper level of \eqref{eq:bi-level_opt}. We consider the scenario where we have access to training data $(u^{(j)},y^{(j)})_{j=1}^n$, which we assume to be i.i.d. samples from  $\mu_{(U,Y)}$. With these data, we can approximate  $F$ in \eqref{eq:bi-level_opt} by its empirical average:
\begin{equation}\label{eq:empirical_loss}
\Fhat_n=\frac{1}{n}\sum_{j=1}^n \mathcal{L}_{\mathcal{U}}(u_\lambda(y^{(j)}),u^{(j)}).
\end{equation}
This leads to a data-driven estimator of the regularization parameter,
\begin{equation}\label{eq:bi-level_empirical}
\begin{split}
\lambdahat_n&= \underset{\lambda\in\Lambda}{\argmin}\ \Fhat_n,\\
u_\lambda(y^{(j)}) &= \underset{u\in\mathcal U}{\argmin}\ \mathcal{L}_{\mathcal{Y}}(\mathcal G(u),y^{(j)}) + \mathcal{S}_{\lambda}(u).
\end{split}
\end{equation}
This method of estimation is often known as empirical risk minimization in machine learning \cite{SB14}. 
We refer to this as ``offline" since minimizing $\Fhat_n$ involves all $n$ data points at each algorithmic iteration.
With $\lambdahat_n$ being formulated, it is of natural interest to investigate its convergence to the true parameter $\lambda_\ast$, when the sample size increases.  Consistency analysis is of central interest in the study of statistics. In particular, if $\lambdahat_n$ is the global minimum of $\Fhat_n$, we have the following theorem 5.2.3 \cite{BD71} from Bickel and Doksum, formulated in our notation
\begin{thm}
\label{thm:Bickel}
Suppose for any $\epsilon>0$
\begin{equation*}
\P(\sup\{\lambda\in \Lambda, |\Fhat_n(\lambda)-F(\lambda)|\}>\epsilon)\to 0,
\end{equation*}
as $n\to\infty$, $\lambdahat_n$ is the global minimizer of $\Fhat_n$,
and $\lambda_\ast$ is the unique minimizer of $F$, then $\lambdahat_n$ is a consistent estimator.
\end{thm}

In more practical scenarios, the finding of $\lambdahat_n$ relies on the choice of optimization algorithms. If we are using gradient based algorithms, such as gradient descent,  $\lambdahat_n$ can be the global minimum of $\Fhat_n$ if $\Fhat_n$ is convex. More generally, we can only assume $\lambdahat_n$ to be a stationary point of $\Fhat_n$, i.e. $\nabla\Fhat_n(\lambdahat_n)=0$. In such situations, we provide the following alternative tool replacing Theorem \ref{thm:Bickel}:

\begin{prop}\label{prop:cond_conv}
Suppose $F$ is $C^2$, $\lambda_\ast$ is a local minimum of $F$, and $\lambdahat_n$ is a local minimum of $\Fhat_n$. 
Let $\mathcal D$ be  an open convex neighborhood of $\lambda_\ast$  in the parameter space and $c_0$ be a positive constant.  Denote $\mathcal{A}_n$ as the event
\[
\mathcal{A}_n=\{\lambdahat_n\in \mathcal D, \nabla^2_\lambda \Fhat_n(\lambda)\succeq c_0 I\text{ for all }\lambda\in \mathcal D\}.
\]
When $\mathcal{A}_n$ takes place, the following holds:
\[
\|\lambdahat_n-\lambda_\ast\|\leq \frac{\|\partial_\lambda \Fhat_n(\lambda_\ast)-\partial_\lambda F(\lambda_\ast)\|}{c_0}.
\]
In particular, we have
\[
\E 1_{\calA_n}\|\lambdahat_n-\lambda_\ast\|\leq  \frac{\sqrt{\text{tr}(\text{Var}(\partial_{\lambda}f(\lambda_\ast,Z)))}}{c_0\sqrt{n}}.
\]
\end{prop}

Proposition \ref{prop:cond_conv} makes two claims. From the second claim, we can see $\lambda_n$ converges to $\lambda_\ast$ at rate of $\frac{1}{\sqrt{n}}$. And with the first claim, sometimes we can have more accurate estimate on large or medium deviations. We will see how to do that in the linear inverse problem discussed below.

On the other hand, Proposition \ref{prop:cond_conv} requires the existence of region $\mathcal D$ so that both $\lambdahat_n$ and $\lambda_\ast$ are in it, moreover $\Fhat_n$ needs to be strongly convex inside $\mathcal D$. The convexity part is necessary, since without it, there might be multiple local minimums, and we will have identifiability issues. 
In order to apply Theorem \ref{prop:cond_conv}, one needs to find $\mathcal D$ and bound the probability of  outlier cases $\calA_n^c$. This procedure can be nontrivial, and requires some advanced tools from probability. We demonstrate how to do so for the linear inverse problem.

\subsection{Offline consistency analysis with linear observation models}
\label{sec:linear_Tikh}

In this section we demonstrate how to apply Proposition \ref{prop:cond_conv} for linear observation models with Tikhonov regularization.  
 
{To motivate our framework,} we assume $u\in \R^d$ and the data $y$ is observed through a matrix $A\in\R^{K\times d}$
\begin{equation}\label{eq:lin_ip}
y = Au + \xi,
\end{equation}
with Gaussian prior information $u\sim\mathcal N(0,\frac1{\lambda_\ast}C_0)$ and Gaussian noise $\xi\sim\mathcal N(0,\Gamma)$. The common choice of discrepancy functions in the lower level are the corresponding negative log-likelihoods
\[
\mathcal{L}_{\mathcal{Y}}(\mathcal{G}(u),y)=\frac12\|Au-y\|_\Gamma^2,\quad \mathcal{S}_{\lambda}(u)=\frac\lambda 2\|u\|_{C_0}^2.
\] 
Since both of these functions are quadratic in $u$, the lower level optimization problem has an explicit solution
\[
u_\lambda(y)=(A^\top\Gamma^{-1}A +\lambda C^{{-1}}_0)^{-1}A^\top \Gamma^{-1}y_i. 
\]
If we use the root-mean-square error in the upper level to learn $\lambda$, the discrepancy function is given by 
\[
f(\lambda,u,y)=\|u_\lambda(y)-u\|^2. 
\]
and the empirical loss function is defined by
\begin{equation}
\label{Fhatlinear}
\Fhat_n(\lambda)=\frac1n\sum_{i=1}^n \|u_\lambda(y_i)-u_i\|^2.
\end{equation}
It is worth mentioning that $F(\lambda)$ is not convex on the real line despite that $\mathcal G$ is linear. The detailed calculation can be found in Remark \ref{rem:app_convexity}. It is of this reason, it is necessary to introduce the local region $\mathcal{D}$ that $F$ is convex inside at Proposition \ref{prop:cond_conv}. 

{In some challenging scenarios the underlying distribution of the noise and parameter is not known. However, as long as one has access to the covariances, $\Gamma$ and $C_0$ of the underlying distribution, the bilevel optimization approach is well-defined and it can still be implemented in order to choose the regularization parameter for the inverse problem. Our theoretical results extend to the general setting, as long as $u$ is $(0,\frac{1}{\lambda_\ast}C_0)$ sub Gaussian, and $\xi$ is $(0,\Gamma)$-sub Gaussian.} 

\begin{assumption}
\label{aspt:subGaussian}
{A random vector $z$ is $(0,\Gamma)$-sub Gaussian if there exists $$ \sigma = (\sigma_1^i,\dots,\sigma_d^{i})^\top\in\R^\ell,$$ such that $$z\overset{d}{=} \Sigma^{1/2} \sigma,$$
where $\sigma_1,\dots, \sigma_\ell$ are i.i.d. random variables with $\E \sigma_1 = 0$, $\E|\sigma_1|^2 = 1$ and $$\sup_{p\ge1}\ p^{-1/2}\E[|\sigma_1|^p]^{1/p}\le C_\sigma,$$
for some $C_\sigma>0$. Furthermore, we assume that the components are symmetric in the sense that $\sigma_1 \overset{d}= -\sigma_1$.}
\end{assumption}

{Sub-Gaussian indicates the the tail of $u$ and $\xi$ are not heavy, so concentration inequalities can be applied \cite{RV13}. Note however that $N(0,\Gamma)$ is $(0,\Gamma)$-sub Gaussian. Uniform distributions $\mathcal{U}[(-a,a)]$ can also be sub-Gaussian.}

{Another important issue in inverse problems is the notion of dimension independence. Since most applications involve models of high or even infinite dimension, it is of interest to see if the parameter recovery depends only on some \emph{effective dimension} but not the ambient space dimension $d$. Here, the effective dimension is usually described by physical quantities in the inverse problem. For this paper, we assume the following:}
\begin{assumption}
\label{aspt:DI}
{
 $\tr(C_0),\|C_0\|_F$ and $\|A^\top\Gamma^{-1}A\|$, $\|A\|$, $\|\Gamma\|$ are  constants independent of the dimension $d$.}
\end{assumption}
{In the subsequent development, we will refer to constants that depend only on  $\tr(C_0),\|C_0\|_F, \|A^\top\Gamma^{-1}A\|$, $\|A\|$, $\|\Gamma\|$  as dimension independent (DI).}

{Roughly speaking, in order for  $\tr(C_0),\|C_0\|_F$ to be DI,  the spectrum of the prior covariance $C_0$ need to decay to zero quickly. By assuming  $\|A^\top\Gamma^{-1}A\|$, $\|A\|$, $\|\Gamma\|$ to be DI, each individual observation needs to be moderately precise. We do not have hard constraints on the number of observations, other than they are independent. All these conditions are practical and can be found in \cite{AMOS19,KS04}. We will also demonstrate they hold for our numerical example(s). %
}

Since the formulation of $u_\lambda$ involves the inversion of matrix $A^\top \Gamma^{-1}A+\lambda C^{-1}_0$,  such an operation may be unstable for $\lambda$ approaching $\infty$. When $\lambda$ approaches $\infty$, the gradient of $\Fhat_n$ approaches zero, so $\infty$ can be a stationary point that an optimization algorithm tries to converge to. To avoid these issues, we assume there are lower and upper bounds
such that  
\[
0<\lambda_l<\frac12\lambda_\ast<\frac32\lambda_\ast<\lambda_u,
\] 
where $\lambda_l$ can be chosen as a very small number and $\lambda_u$ can be very large.
Their values often  can be obtained from physical restrictions from  the inverse problem. By assuming their existence, we can restrict $\lambdahat_n$ to be in the interval $\Lambda=(\lambda_l,\lambda_u)$.  Now we are ready to present our main result for the offline recovery of regularization parameter. In particular, we show $\lambdahat_n$ converges to $\lambda_\ast$ with high probability.  

\begin{thm}
\label{thm:prob}
{Suppose that $u$ is $(0,\frac{1}{\lambda_\ast}C_0)$-sub Gaussian, $\xi$ is $(0,\Gamma)$-sub Gaussian, %and $\xi$ with covariance matrices $ \Cov(u) = \frac1{\lambda_\ast}C_0$ and  $\Cov(\xi) = \Gamma$ 
where $C_0\in\R^{d\times d}$, $\Gamma\in\R^{K\times K}$ are known symmetric positive definite matrices and $\lambda_\ast>0$ is unknown, and let} $\lambda_n^\ast\in (\lambda_l,\lambda_u)$ be a local minimum of $\Fhat_n$. {Then it holds true that $\partial_\lambda F(\lambda_\ast)=0$ and} there exist $C_\ast,c_\ast>0$ such that for any $1>\epsilon>0$  and $n$, 
\begin{align*}
\bbP(|\hat{\lambda}_n-\lambda_\ast|>\epsilon,\lambda_l<\lambdahat_n<\lambda_u)\leq C_\ast\exp(-c_\ast n\min(\epsilon,\epsilon^2)).
\end{align*}
The values of $C_\ast,c_\ast>0$ depend on $\lambda_l,\lambda_u,\lambda_\ast,C_0$ but not on $n$.   Moreover, if Assumption \ref{aspt:DI} is assumed, $C_\ast,c_\ast>0$ are also dimension independent. 
\end{thm}
Since we can obtain consistency assuming that $\lambdahat_n$ is a local minimum, we do not demonstrate how to implement Theorem \ref{thm:Bickel} for the more restrictive scenario where $\lambdahat_n$ is a global minimum.

{In general, knowing that $\lambdahat_n$ is an accurate estimator is sufficient to guarantee the recovery is accurate, also due to the Lipschitzness of $u_\lambda$. In particular, we can also show the Lipschitz constant is DI: }
\begin{prop}
\label{prop:Lip}
{Under Assumption \ref{aspt:DI}, there is a dimension independent $L$ so that 
\[
\E_y\|u_\lambda(y)-u_{\lambda_\ast}(y)\|^2\leq L |\lambda-\lambda_\ast|^2. 
\]
Here $y$ is a random observation sample.}
\end{prop}

\begin{remark}
{In practical settings it is assumed that the noise is not fully known and we can easily extend our results by changing $\Gamma \mapsto\frac1\gamma \Gamma$. We then would try to choose the ration between regularization parameter and noise scale, this is, we can change $\lambda\mapsto \frac{\lambda}{\gamma}$ and apply again the bilevel optimization approach with ''known" noise covariance $\Gamma$.}
\end{remark}

\section{Regularization parameter online recovery}
\label{sec:reg}
In this section we discuss how to implement the stochastic gradient descent(SGD) method for online solutions of the bilevel optimization. We will formulate the SGD method for general nonlinear inverse problems and state certain assumptions on the forward model and the regularization function to ensure convergence of the proposed method.

\subsection{Bilevel stochastic gradient descent method}
In the offline solution of the bilevel optimization problem \eqref{eq:bi-level_empirical},
one has to compute the empirical loss function $\Fhat_n$ or its gradient in \eqref{eq:empirical_loss}. This involves
solving the lower level problem for each training data point $(u^{(j)},y^{(j)})$, $j=1,\dots,n$. When $n$ is very large, this can be  very computationally demanding.  One way to alleviate this is to use the stochastic gradient descent (SGD). This has been done in the context of traditional optimization \cite{CW18}, where various convergence results were shown. As a result this has been applied to problems in machine learning, most notably support vector machines \cite{CW14,CW17}, but also in a more general context without the use of SGD \cite{FFS18,JF18}.
Here we propose a SGD method to solve the bilevel optimization problem \eqref{eq:bi-level_opt} online.

{To formulate the SGD method, we first note that the gradient descent method generates iterates $\lambda_{k+1}$ based on the following update rules
\[
\lambda_{k+1}=\lambda_k-\beta_k \partial_\lambda F(\lambda_k),
\]
where $\beta_k$ is a sequence of stepsizes.} 

As mentioned above, the population gradient $\partial_\lambda F$ is often computationally inaccessible, and its empirical approximation $\partial_\lambda \Fhat_n$ is often expensive to compute. One general solution to this issue is using a stochastic approximation of $\partial_\lambda F$. Here we choose $\partial_\lambda f(\lambda_k, Z^{(k)})$, since it is an unbiased estimator of $\partial_\lambda F$:
\[
\partial_\lambda F(\lambda_k)=\E_{Z} \partial_\lambda f(\lambda_k,Z).
\]
The identity above holds by Fubini's theorem, since we assume $f$ and its second order  derivatives are all continuous and differentiable.  Comparing with $\partial_\lambda \Fhat_n$, $\partial_\lambda f$ involves only one data point $Z^{(k)}$, so it has a significantly smaller computation cost. We refer to this method as ``online", since it does not require all $n$ data points available at each algorithmic iteration.

We formulate the stochastic gradient descent method to solve \eqref{eq:bi-level_opt} as Algorithm \ref{alg:SGD}.

\begin{algorithm}[H]
\caption{Bilevel Stochastic Gradient Descent}
\label{alg:SGD}
\begin{algorithmic}[1]
\State Input: $\lambda_0$, {$m,$} $\beta = (\beta_k)_{k=1}^{n}$, $\beta_k>0$, i.i.d. sample $(Z^{(k)})_{k\in\{1,\dots,n\}}\sim\mu_{(U,Y)}$.
\For{$k = 0,\dots,n-1$} 
\begin{equation}\label{eq:SGD_general}
\lambda_{k+1} = \chi(\lambda_k -\beta_k \partial_\lambda f(\lambda_k,Z^{(k)})),
\end{equation}
\EndFor
\State {Output: the average $\bar{\lambda}_n=\frac{1}{m}\sum_{k=n-m+1}^n \lambda_k$}
\end{algorithmic}
\end{algorithm}
In Algorithm \ref{alg:SGD},
the step size $\beta_k$ is a sequence decreasing to zero not too fast, so that the Robbins--Monro conditions \cite{RM51} apply:
\begin{equation}\label{eq:RM_cond}
\sum\limits_{k=1}^\infty \beta_k = \infty,\quad \sum\limits_{k=1}^\infty \beta_k^2 <\infty.
\end{equation}
One standard choice is to take a decreasing step size $\beta_k=\beta_0 k^{-\alpha}$ with $\alpha\in (1/2,1]$. {We note that the output of our bilevel SGD method is given by the average over the last iterations $\bar\lambda_n$, which has been shown to accelerate the scheme for standard SGD methods, see \cite{PJ92}}. The projection map $\chi$ (\cite{SB14} Section 14.4.1) is defined as
\[
\chi(\lambda)=\arg\min_{\theta\in \Lambda}\{\|\theta-\lambda\|\}.
\]
In other words, it maps $\lambda$ to itself if $\lambda\in \Lambda$, otherwise it outputs the point in $\Lambda$ that is closest to $\lambda$.  Using $\chi$  ensures $\lambda_{k+1}$ is still in the range of regularization parameter if $\Lambda$ is closed. This operation in general shorten the distance between  $\lambda_{k+1}$ and $\lambda_\ast$ when $\Lambda$ is convex:

\begin{lem}[Lemma 14.9 of \cite{SB14}]
\label{lem:convball}
If $\Lambda$ is convex, then for any $\lambda$ 
\[
\|\chi(\lambda)-\lambda_\ast\|\leq\|\lambda-\lambda_\ast\|.
\]
\end{lem}

In particular, the stochastic gradient $\partial_\lambda f(\lambda_k,Z^{(k)})$ is given by the following lemma, which states sufficient conditions on $\Psi$ to ensure both $u_\lambda$ and $f$ are continuously differentiable w.r.t. $\lambda$.

\begin{lem}
\label{lem:grad_reg}
Suppose the lower level loss function $\Psi(\lambda,u,y)$
{is $C^2$ for $(u,\lambda)$ in a neighborhood of $(u_{\lambda_0},\lambda_0)$ and is strictly convex in $u$ in this neighborhood}, then the function $\lambda\mapsto u_\lambda(y)$ is continuously differentiable w.r.t. $\lambda$ near $\lambda_0$ and the derivative is given by
\begin{equation}\label{eq:tik_grad}
\partial_\lambda u_\lambda(y) = -\left(\nabla_{u}^2 \left[\Psi(\lambda,u_\lambda(y),y)\right]\right)^{-1}\partial_{\lambda u}^2\left[\Psi(\lambda,u_\lambda(y),y)\right].
\end{equation}
and
\begin{equation}\label{eq:bi-level_grad}
\partial_\lambda f(\lambda,y,u)=\partial_w\mathcal{L}_{\mathcal{U}}(u_\lambda(y),u)^T \partial_\lambda u_\lambda(y).
\end{equation}
\end{lem}

\subsection{Approximate stochastic gradient method}
In order to implement Algorithm \ref{alg:SGD}, it is necessary to evaluate the gradient $\partial_\lambda f$. While 
Lemma \ref{lem:grad_reg} provides a formula to compute the gradient, its evaluation can be expensive for complicated PDE forward models. In these scenarios, it is more reasonable to implement approximate SGD. 

One general way to find approximate gradient is applying central finite difference schemes. This involves perturbing certain coordinates in opposite direction, and use the value difference to approximate the gradient:
\begin{equation}
\label{eq:genaGD}
\widetilde\partial_\lambda f(\lambda_k,z)\approx
\frac{f(\lambda_k+h_k,z)-f(\lambda_k-h_k,z)}{2h_k},
\end{equation}
where $h_k$ is a step size. The step size $h_k$ can either be fixed as a small constant, or it can be decaying as $k$ increases, so that higher accuracy gradients are used when the iterates are converging.

In many cases, the higher level optimization uses a $L_2$ loss function
\begin{equation*}
\mathcal L_{\mathcal U}(y,u) = \|y-u\|^2.
\end{equation*}
The exact SGD update step \eqref{eq:SGD_general} can be written as
\begin{align*}
\lambda_{k+1} 	&= \lambda_k -\beta_k \partial_{\lambda}\|u_{\lambda_k}(y^{(k)})-u^{(k)}\|^2\\
							&= \lambda_k -\beta_k \left(\partial_{\lambda}u_{\lambda_k}(y^{(k)})\right)^\top (u_{\lambda_k}(y^{(k)})-u^{(k)}).
\end{align*}
In this case, it makes more sense to apply central difference scheme only on the $\partial u_\lambda$ part:
\begin{equation}\label{eq:grad_approx}
{\partial_\lambda u_\lambda(y^{(k)}) \approx \frac{ u_{\lambda+{h_k}}(y^{(k)})-u_{\lambda-{h_k}}(y^{(k)})}{2h_k}=: \widetilde\partial_\lambda u_\lambda(y^{(k)}).}
\end{equation}
Using this approximation, we formulate the approximate SGD method in the following algorithm, where we replace the exact gradient $\partial_\lambda u_\lambda(y^{(k)})$ by the numerical approximation $\widetilde\partial_\lambda u_\lambda(y^{(k)})$ defined in \eqref{eq:grad_approx}.

Here we have defined the numerical approximation of $\partial_\lambda f$ by
\begin{equation}
\label{eq:L2aGD}
\widetilde\partial_\lambda f(\lambda,(y,u)):=\left(\widetilde\partial_{\lambda}u_{\lambda}(y)\right)^\top (u_{\lambda}(y)-u).
\end{equation}
{In most finite difference approximation schemes, the approximation error involved is often controlled by $h_k$. In particular, we assume the centred forward difference scheme used in either \eqref{eq:genaGD} or \eqref{eq:L2aGD} yields an error of order
\[
\|\E \widetilde\partial_\lambda(f(\lambda,Y,U))-\partial_\lambda F(\lambda)\|=:\alpha_k=O(h^2_k).
\]
Replacing the stochastic gradient in Algorithm \ref{alg:SGD} with its approximation, we obtain the algorithm below:
}

\begin{algorithm}[H]
\caption{Approximate Bilevel Stochastic Gradient Descent}
\label{alg:aBSGD}
\begin{algorithmic}[1]
\State Input: $\lambda_0$, $m,$ $\beta = (\beta_k)_{k=1}^{n}$, $\beta_k>0$, i.i.d. sample $(Z^{(k)})_{k\in\{1,\dots,n\}}\sim\mu_{(U,Y)}$.
\For{$k = 0,\dots,n-1$} 
\begin{equation}\label{eq:SGD_empirical_approx}
\lambda_{k+1} =  \chi(\lambda_k -\beta_k \widetilde\partial_\lambda f(\lambda_k,Z^{(k)})),
\end{equation}
\EndFor
\State {Output: the average $\bar{\lambda}_n=\frac{1}{m}\sum_{k=n-m+1}^n \lambda_k$}
\end{algorithmic}
\end{algorithm}

\subsection{Consistency analysis for online estimators}
Next we formulate sufficient conditions that can ensure that $\lambda_{k}$ converges in $L^2$ to the optimal solution $\lambda_{\ast}$ of \eqref{eq:bi-level_opt}.
\begin{prop}\label{prop:gensgd}
Suppose that there is a convex region $\mathcal{D}\subset \Lambda$ and a constant $c>0$ such that 
\begin{equation}
\label{eqn:A2}
\inf_{\lambda\in \mathcal{D}}\ (\lambda-\lambda_\ast) \partial_\lambda F(\lambda) >c\|\lambda-\lambda_\ast\|^2.
\end{equation}
and there are constants $a,\ b>0$ such that for all $\lambda\in \mathcal{D}$ it holds true that 
\begin{equation}
\label{eqn:A4}
\E[|\widetilde{\partial}_\lambda f(\lambda,Z)|^2] < a+b\|\lambda-\lambda_\ast\|^2.
\end{equation}
Also the bias in the approximated SGD is bounded by 
\begin{equation}
\label{eqn:A3}
\|\E \widetilde{\partial}_\lambda f(\lambda_k,Z_k)-\partial_\lambda F(\lambda_k)\|^2\leq \alpha_k. 
\end{equation}
Let $\mathcal{A}_n$  be the event that $\lambda_k\in \mathcal{D}$. Suppose $\beta_0\leq \frac{c}{b}$. Then if the approximation error is bounded by a small constant $\alpha_k\leq \alpha_0$, there is a constant $C_n$ such that 
\[
\E 1_{\calA_n}\|\lambda_n-\lambda_\ast\|^2
\leq  \left(\E Q_0+2a\sum_{j=1}^\infty \beta_j^2\right)C_n+\frac{\alpha_0}{c^2}.
\]
Here 
\begin{equation}
\label{eqn:Cn}
C_n=\min_{k\leq n}\max\left\{\prod_{j=k+1}^n(1-c\beta_j), a\beta_k/c\right\} 
\end{equation}
is a sequence converging to zero. 

If the approximation error is decaying so that $\alpha_k\leq D\beta_k$, then we have the estimation error
\[
\E 1_{\calA_n}\|\lambda_n-\lambda_\ast\|^2
\leq \left(\E Q_0+2(a+D/c)\sum_{j=1}^\infty \beta_j^2\right)C_n.
\]
\end{prop}

{
\begin{remark}
We note that the above result also leads to similar convergence of the average estimator $\bar \lambda_n$ since by Jensen's inequality
{
\[
\|\bar{\lambda}_n-\lambda_\ast\|^2\leq \frac{1}{m}\sum_{k=n-m+1}^n
\|\lambda_k-\lambda_\ast\|^2.
\] }
Further, for standard SGD methods the averaging step has been shown to lead to the highest possible convergence rate under suitable assumptions. Interested readers can refer to \cite{PJ92} for more details.
\end{remark}
}

\subsection{Consistency analysis with linear inverse problem}

{We consider again the linear inverse problem from Section~\ref{sec:linear_Tikh}
\begin{equation*}
y = Au+\xi,
\end{equation*}
but we do not state specific assumptions on the distribution without $\E[|u|^4]<\infty$ and $\E[|A^\top\Gamma^{-1}\xi|^4]<\infty$.}

{We formulate the online convergence for the corresponding bilevel optimization with least-square data misfit and Tikhonov regularization, i.e.
\begin{equation*}
\mathcal L_{\mathcal Y}(Au,y) = \frac12\|Au-y\|_{\Gamma}^2, \quad S_\lambda(u) =\frac{\lambda}{2}\|u\|_{C_0}^2.
\end{equation*}
}

\begin{thm}
\label{thm:linSGD}
Let $\beta=(\beta_k)_{k\in\N}$ be a sequence of step sizes with $\beta_k>0$, $\sum\limits_{k=1}^\infty \beta_k=\infty,$ and $\sum\limits_{k=1}^\infty \beta_k^2<\infty.$ {Furthermore, let $\E[|u|^4]<\infty$ and $\E[|A^\top\Gamma^{-1}\xi|^4]<\infty$.}
Then for some constant $B$ and a sequence $C_n$ converging to zero, the following hold
\begin{enumerate}
\item the iterates generated from the exact SGD, Algorithm~\ref{alg:SGD}, converge to $\lambda_\ast$ in the sense
\[
\E\|\lambda_n-\lambda_\ast\|^2\leq B C_n,
\]
\item the iterates generated from the aproximate SGD, Algorithm~\ref{alg:aBSGD} with formula \eqref{eq:L2aGD} and $h_k=h$, converge to $\lambda_\ast$ up to an error of order $\mathcal{O}(h^4)$, i.e.
\[
\E\|\lambda_n-\lambda_\ast\|^2\leq B(C_n+h^4).
\]
If we use decaying finite difference stepsize $h_k\leq h\beta^{1/4}_k$, then the error can be further bounded by 
\[
\E\|\lambda_n-\lambda_\ast\|^2\leq B C_n.
\]
\end{enumerate}
Moreover, if Assumption \ref{aspt:DI} is assumed, the constants $B,C_n$ are  dimension independent.  

\end{thm}
{
\begin{remark}
While in the offline setting the proof of the consistency result for the linear {sub-Gaussian} setting was heavily relying on the {sub-Gaussian} assumption on $u$ and $\xi$, in the online setting we are able to extend the result {to a wider range of distributions} of $u$ and $\xi$. For our proof of Theorem~\ref{thm:linSGD} in Appendix~\ref{app:proof_linear_online} we only need to assume that $\E[|u|^4]<\infty$ and $\E[|A^\top\Gamma^{-1}\xi|^4]<\infty$. Hence, it can also be applied to general linear inverse problems without {sub-Gaussian} assumption on the unknown parameter or {sub-Gaussian} assumption on the noise.
\end{remark}
}

\section{Numerical results}
\label{sec:num}

In this section our focus will be directed on testing the results of the previous sections. We will present various inverse problems to our theory, which will be based on partial differential equations, both linear and nonlinear which includes a linear 2D Laplace equation, a 2D Darcy flow from geophysical sciences and a 2D eikonal equation which arises in wave propagation. As a final numerical experiment, related to the examples discussed in Section \ref{sec:intro}, we test our theory on an image denoising problem.

For the linear example, we have access to the exact derivative of the Tikhonov solution for the {bilevel optimization. In particular, we can implement both offline and online bilevel optimization methodologies. In contrast, finding the exact derivatives for nonlinear inverse problems  is difficult both in derivation and computation, so we will only use online methods with approximated gradient. For online methods, we implement the following variants:
\begin{itemize}
\item $\text{bSGD}$: Application of the bilevel SGD, Algorithm~\ref{alg:SGD} with exact derivative \eqref{eq:bi-level_grad}.
\item $\text{bSGD}_{a}$: Application of the bilevel SGD, Algorithm~\ref{alg:aBSGD} with derivative approximation \eqref{eq:L2aGD} for fixed $h_k=h_0$ in \eqref{eq:grad_approx}.
\end{itemize}
{For our first model we have tested both bSGD and $\text{bSGD}_{a}$, while for the nonlinear models we have used $\text{bSGD}_{a}$. It is worth mentioning that we have also tested, as a side experiment, using the adaptive derivative $h_k=h_0/k^{1/4}$. For these experiments it was shown that the adaptive derivative scheme does not show any major difference to the case of fixed $h_k=h_0$. {In fact, Theorem \ref{thm:linSGD} has already implied this, since the difference between the two scheme is of order $h_0^{-4}$, which is often smaller than the error from the numerical forward map solver or the use of $\bar{\lambda}_n$. }
For this reason, we  do not demonstrate this scheme in our numerics. 
 }

\subsection{Linear example: 2D Laplace equation}
\label{ssec:linear_num}
We consider the following forward model 
\begin{equation}\label{eq:Laplace}
\begin{cases}
-\Delta p(x) &= \ u(x), 	\quad x\in {X},\\
\ \ \ \ \ p(x) &=\  0,			\quad	\quad \ 	x \in \partial {X},
\end{cases}
\end{equation}
with Lipschitz domain $ {X}=[0,1]^2$ and consider the corresponding inverse problem of recovering the unknown $u^\dagger$ from observation {of \eqref{eq:Laplace}, described through
\begin{equation}
y = \mathcal O(p) + \eta,
\end{equation}
where $\eta\sim\mathcal N(0,\Gamma)$ is measurements noise and $p$ is the solution of \eqref{eq:Laplace}. We solve the PDE in weak form, where} $\mathcal A^{-1}:\mathcal U\to\mathcal V$, with $\mathcal U=L^\infty( {X})$ and $\mathcal V = H^1_0( {X})\cap H^2( {X})$, denotes the solution operator for \eqref{eq:Laplace} and $\mathcal O:\mathcal V\to\R^K$ denotes the observation map taking measurements at $K$ randomly chosen points in $ {X}$, i.e.~$\mathcal O(p) = (p(x_1),\dots,p(x_K))^\top$, for $p\in\mathcal V$, $x_1,\dots,x_K\in  {X}$. For our numerical setting $K=250$ points have been observed, which is illustrated in Figure~\ref{fig:obs_points}.  We can express this problem as a linear inverse problem in the reduced form \eqref{eq:lin_ip} by
\begin{equation}
y^\dagger = Au^\dagger+\eta \in\R^K,
\end{equation}
where $A = \mathcal O\circ \mathcal A^{-1}$ is the forward operator which takes measurements of \eqref{eq:Laplace}. The forward model \eqref{eq:Laplace} is solved numerically on a uniform mesh with $1024$ grid points in $ {X}$ by a finite element method with continuous, piecewise linear finite element basis functions.

We assume that our unknown parameter $u^\dagger$ follows a Gaussian distribution $\mathcal N(0,\frac1{\lambda_\ast}C_0)$ with covariance
\begin{equation}\label{eq:cov}
C_0 = \beta\cdot ( \tau^2 I-\Delta)^{-\alpha},
\end{equation}
with Laplacian operator $\Delta$ equipped with Dirichlet boundary conditions, known $\beta,\ \tau >0$, $\alpha>1$ and unknown $\lambda_\ast>0$. To sample from the Gaussian distribution, we consider the truncated Karhunen-Lo\`{e}ve (KL) expansion \cite{LPS14}, which is a series representation for $u\sim \mathcal N(0,C_0)$, i.e.
\begin{equation}
\label{eq:KLE}
u(x) = \sum\limits_{i=1}^\infty \xi_i\sqrt{\frac{1}{\lambda_\ast}\sigma_i}\varphi_i(x),
\end{equation}
where $(\sigma_i,\varphi_i)_{i\in\N}$ are the eigenvalues and eigenfunction of the covariance operator $C_0$ and $\xi=(\xi_i)_{i\in\N}$ is an i.i.d. sequence with $\xi_1\sim\mathcal N(0,1)$ i.i.d.\,. Here, we have sampled from the KL expansion for the discretized $C_0$ on the uniform mesh. Furthermore, we assume to have access to training data $(u^{(j)},y^{(j)}=Au^{(j)}+\eta^{(j)})_{j=1,\dots,n}$, $n\in\N$, which we will use to learn the unknown scaling parameter $\lambda_\ast$ before solving the inverse problem. For the numerical experiment we set $\beta = 100,\ \tau = 0.1,\ \alpha = 2$ and $\lambda_\ast = 0.1$. After learning the regularization parameter, we will compare the estimated parameter through the different results of the Tikhonov minimum
\begin{equation*}
u_{\lambda_i}(y^\dagger) = (A^\top\Gamma^{-1}A+\lambda_i\cdot C_0)^{-1}A^\top y^\dagger,
\end{equation*}
for $\lambda_1 = \widehat\lambda$ learned from the training data, $\lambda_2 = \lambda_\ast$ and fixed $\lambda_3 = 1$. We have used the MATLAB function \verb+fmincon+ to recover the the regularization parameter offline by solving the empirical optimization problem
\begin{equation}
\widehat \lambda_n \in \underset{\lambda>0}{\argmin}\ \frac1n \sum\limits_{j=1}^n |u_\lambda(y^{(j)})-u^{(j)}|^2.
\end{equation}
We use $M=1000$ samples of training data to construct Monte--Carlo estimates of $\E[|\widehat\lambda_n-{\lambda_\ast}|^2]$.
While the computation of the empirical loss function can be computational demanding, we also apply the proposed online recovery in form of the SGD method to learn the regularization parameter $\lambda$ 
by running Algorithm~\ref{alg:SGD} with chosen step size $\beta_k=200/k$, range of regularization parameter $\Lambda = [0.0001, 10]$ and initial value $\lambda_0=1$. 
The resulting iterate $\lambda_k$ can be seen in Figure \ref{fig:offline_online} on the right side.

\begin{figure}[!htb]
	\begin{center}
	\includegraphics[width=0.6\textwidth]{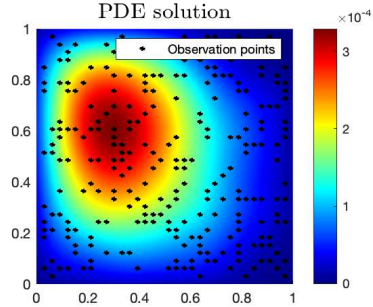}
    \caption{Reference PDE solution for the Laplace equation of the underlying unknown parameter $u^\dagger$, and the corresponding randomized observation points $x_1,\dots,x_K\in  {X}$.}\label{fig:obs_points}
    \end{center}
\end{figure} 

\begin{figure}[!htb]
	\begin{subfigure}[b]{0.49\textwidth}
\includegraphics[width=1.1\textwidth]{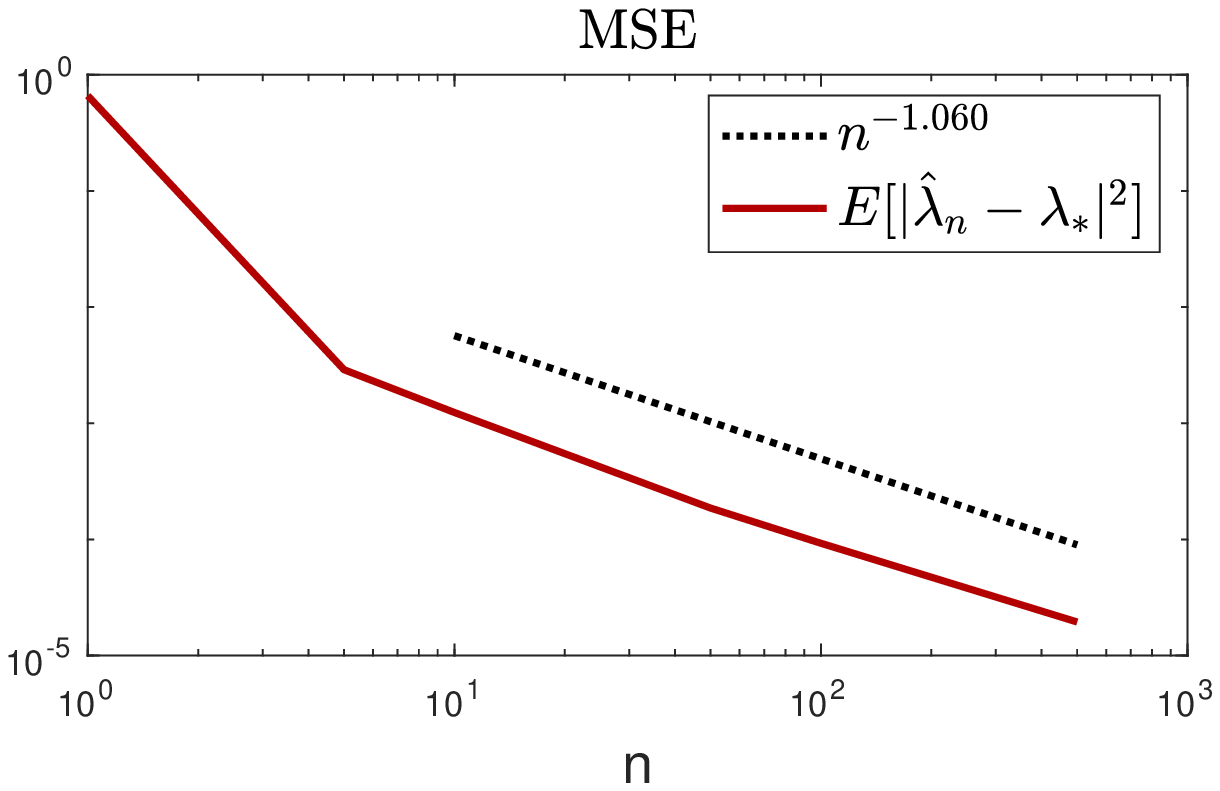}
	\end{subfigure}
	\begin{subfigure}[b]{0.49\textwidth}
	\includegraphics[width=1.1\textwidth]{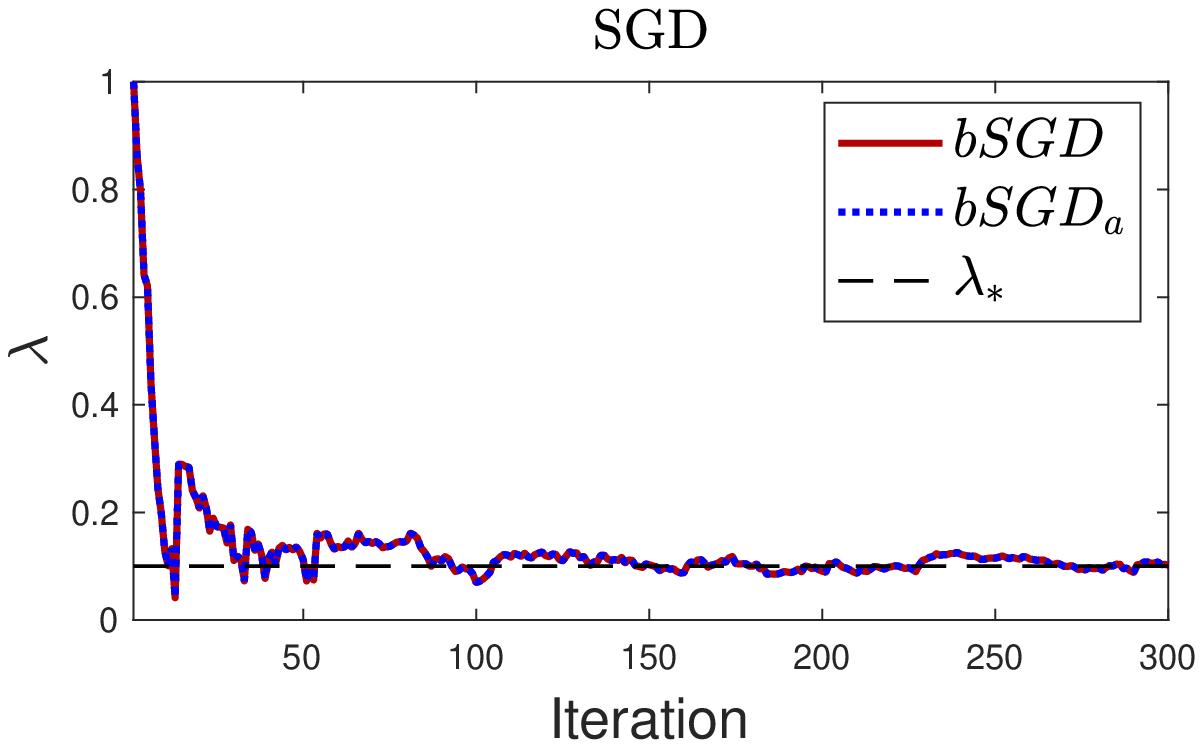}
	\end{subfigure}
    \caption{MSE (left) resulting from the offline recovery depending on training data size. Learned regularization parameter $\lambda_k$ (right) resulting from the online recovery, Algorithm~\ref{alg:SGD} for the Laplace equation.}\label{fig:offline_online}
\end{figure} 

\begin{figure}[!htb]
	\begin{center}
	\includegraphics[width=0.9\textwidth]{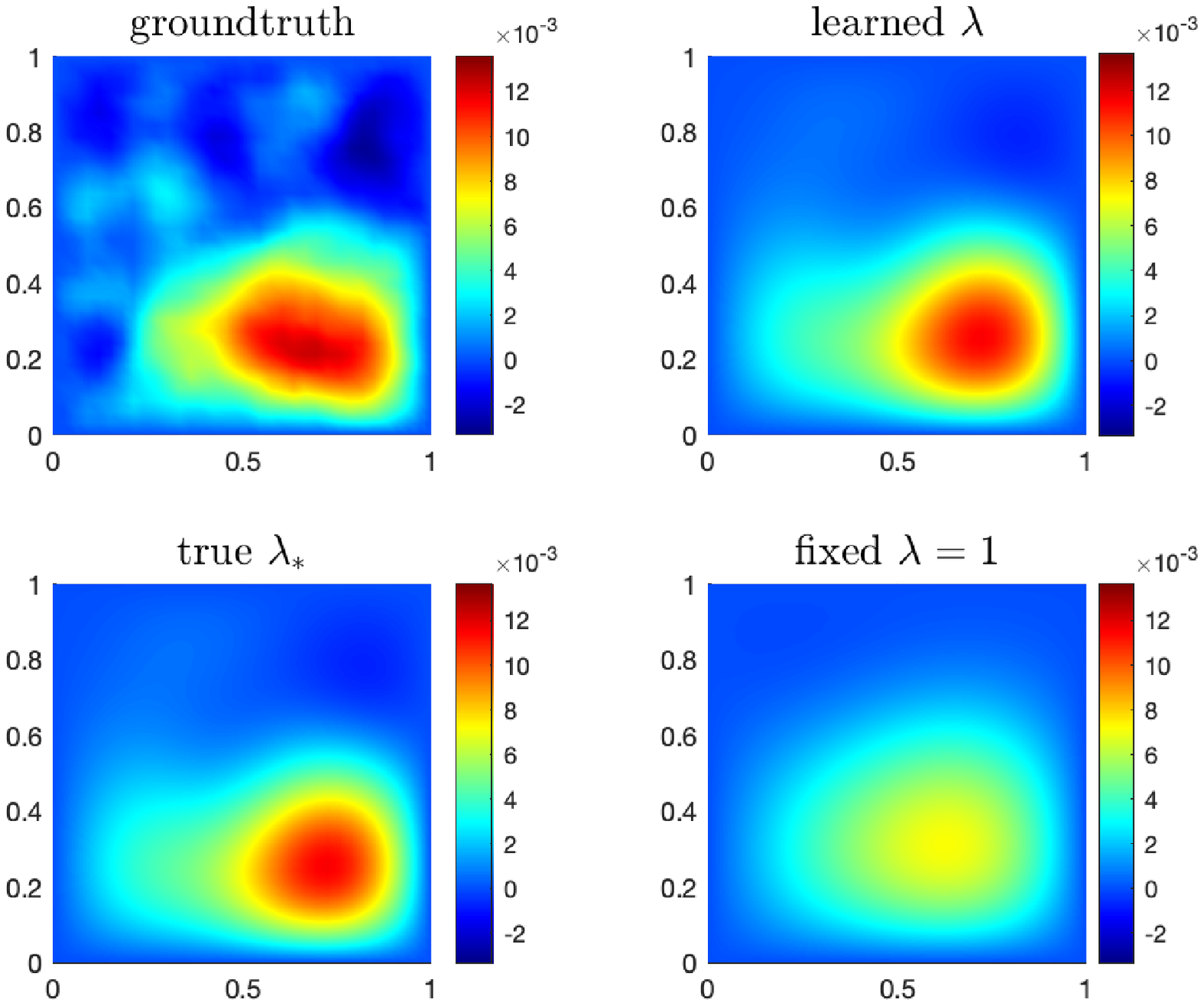}
    \caption{Comparison of different Tikhonov solutions for choices of regularization parameter $\lambda_i$. The learned Tikhonov regularized solution corresponds to the resulting one of the SGD method Algorithm \ref{alg:SGD} for the Laplace equation.}\label{fig:Tik_min}
    \end{center}
\end{figure} 

From the numerical experiments for the linear example we observe that the numerics match our derived theory. In the offline recovery setting, this is first evident in Figure \ref{fig:offline_online} on the left side. We compare the MSE with the theoretical rate, which seems to decay at the same rate. The online recovery is highlighted by the right plot in Figure \ref{fig:offline_online} which demonstrates the convergence towards $\lambda_{\ast}$ as the iterations progress. Further, we show the result of the approximate bSGD method Algorithm~\ref{alg:aBSGD} for fixed chosen $h_k=0.01$ in \eqref{eq:grad_approx}. As the derivative approximation \eqref{eq:grad_approx} is closely exact, we see very similar good performance of the approximate bSGD method.

Finally, Figure \ref{fig:Tik_min} shows the recovery of the underlying unknown through different choices of $\lambda$. It verifies that the adaptive learning of $\lambda$ outperforms that of fixed regularization parameter $\lambda=1$.

\subsubsection{Dimension independent experiments}

{
Next, we are going to analyze the indepence of dimension in the bilevel optimization approach. Our setup is similar as discussed before, but considering the domain $X=[0,1]$.}

{
We solve the forward model numerically on a uniform mesh for different choices of mesh sizes $h\in\{2^{-5},2^{-6},2^{-7},2^{-8}\}$ by a finite element method with continuous, piecewise linear ansatz functions, where the same $K=5$ observation points have been observed on each mesh. We assume that the underlying parameter $u^\dagger$ follows a Gaussian distribution $\mathcal N(0,\frac1\lambda_\ast C_0)$ with $C_0 = (-\Delta)^{-1}$ and apply again the truncated KL expansion up to a fixed truncation index, but considering discretized versions $\varphi_i^h$ of $\varphi_i(x)$ on each level $h$.}

{
 In Figure~\ref{fig:comp}, we compare the MSE for the resulting estimates $\widehat\lambda_n$ for different choices of sample size $n$ depending on the dimension $d$. Here, we use again $M=1000$ samples of training data to construct Monte-Carlo estimates of the MSE $\E[|\widehat\lambda_n-\lambda_\ast|^2]$.}

\begin{figure}[h!]
\centering
\includegraphics[width=0.5\linewidth]{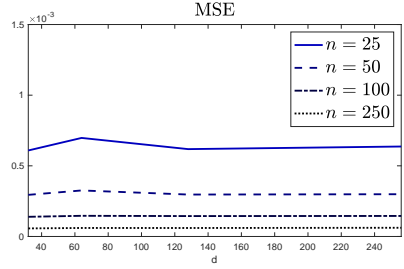}
\caption{{MSE resulting from the offline recovery depending on dimension of the parameter space $d$ and evaluated for different choices of training data size $n$.}}
 \label{fig:comp}
\end{figure}

\subsection{Nonlinear example: 2D Darcy flow}
\label{ssec:nlinear}

We now consider the following elliptic PDE which arises in the study of subsurface flow known as Darcys flow. The forward model is concerned using the log-permeability $\log u \in L^{\infty}( X)=: \mathcal{U}$ to solve for the pressure $p \in H^1_0( {X})\cap H^2({X})=: \mathcal V$ from
\begin{equation}\label{eq:df}
\begin{cases}
-\nabla\cdot(\exp(u) \nabla p) &= \ f, 	\quad x\in  {X}\\
 \quad \quad \quad \quad \quad \quad \  p &=\ 0,				\quad	x \in \partial  {X}
\end{cases}
\end{equation}
with domain $ {X}=[0,1]^2$ and known scalar field $f \in \R$. 
{We again consider the corresponding inverse problem of recovering the unknown $u^\dagger$ from observation of \eqref{eq:df}, described through
\begin{equation}
y = \mathcal{O}(p)+\eta,
\end{equation}
where $\mathcal O:\mathcal{V}\to \R^K$ denotes the linear observation map, which takes again measurements at $K$ randomly chosen points in $ {X}$, i.e.~$\mathcal O(p) = (p(x_1),\dots,p(x_K))^\top$, for $p\in\mathcal{V}$, $x_1,\dots,x_K\in  {X}$.  For our numerical setting we choose $K=125$ observational points, which can again be seen in Figure~\ref{fig:obs_points2}. The measurements noise is denoted by $\eta\in\mathcal N(0,\Gamma)$, for $\Gamma\in\R^{K\times K}$ symmetric and positive definite.

{We formulate the inverse problem through
\begin{equation}\label{eq:ip_nonlinear}
y^\dagger =  \mathcal{G}(u^\dagger)+\eta,
\end{equation}
with $ \mathcal{G} = \mathcal O\circ G$}, where $G:\mathcal U\to\mathcal{V}$ denotes the solution operator of \eqref{eq:df}, solving the PDE \eqref{eq:df} in weak form. The forward problem \eqref{eq:df} has been solved by a second-order centered finite difference method on a uniform mesh with $256$ grid points.

We assume that $u^\dagger$ follows the Gaussian distribution $\mathcal N(0,\frac1{\lambda_\ast}C_0)$ with a covariance operator \eqref{eq:cov} prescribed with Neumann boundary condition. Similar as before, $\beta, \tau>0$ and $\alpha>1$ are known, while $\lambda_\ast>0$ is unknown. This time, in order to infer the unknown parameter, we use the KL expansion and do estimation of the coefficients $\xi$. {See also \cite{CIRS18,GHLS2019} for more details.}
Therefore we truncate \eqref{eq:KLE} up to $d$ and consider the nonlinear map $\mathcal G:\R^{d}\to\R^K$, with $\mathcal G(\xi) = G(u^\xi(\cdot))$ and
$$u^\xi(\cdot) = \sum_{i=1}^d \xi_i\sqrt{\frac1{\lambda_\ast}\sigma_i}\varphi_i(\cdot).$$ 
This implies our unknown parameter is given by $\xi\in\R^{d}$ and we set a Gaussian prior on $\xi$ with $\mathcal N(0,\frac{1}{\lambda_\ast}I)$, where $\lambda_\ast>0$ is unknown.

We again assume to have access to training data $(\xi^{(j)},y^{(j)})_{j=1,\dots,n}$, $n\in \mathbb N$, where $\xi^{(j)}\sim\Xi\sim\mathcal N(0,\frac1{\lambda_\ast}I)$ and we aim to solve the original bilevel optimization problem
\begin{equation}
\widehat\lambda \in\underset{\lambda>0}\argmin\ \E[ \|u_\lambda(Y) - \Xi\|^2], \quad
u_\lambda(Y) = \underset{\xi\in\R^{d}}{\argmin}\ \frac12\| \mathcal G(\xi)-Y\|_{\Gamma}^2 + \frac{\lambda}2\|\xi\|_{I}^2.
\end{equation}
The corresponding empirical optimization problem is given by
\begin{equation}\label{eq:nonlinear_bi-level}
\widehat\lambda^n \in\underset{\lambda>0}\argmin\ \frac1n\sum\limits_{j=1}^n \|u_\lambda(y^{(j)}) - \xi^{(j)}\|^2,\quad
u_\lambda(y^{(j)}) = \underset{\xi\in\R^{d}}{\argmin}\ \frac12\| \mathcal G(\xi)-y^{(j)}\|_{\Gamma}^2 + \frac{\lambda}2\|\xi\|_{I}^2,
\end{equation}
for a given size of the training data $n$. In comparison to the linear setting, we are not able to compute the Tikhonov minimum analytically for each observation $y^{(j)}$, as we require more computational power to solve  \eqref{eq:nonlinear_bi-level}. We will solve \eqref{eq:nonlinear_bi-level} online by application of Algorithm~\ref{alg:aBSGD}, where we will approximate the derivative of the forward model by centered different method \eqref{eq:grad_approx}. We keep the accuracy of the numerical approximation fixed to $h_k=0.01$.

For our numerical results we choose $d=25$
coefficients in the KL expansion and the noise covariance $\Gamma = \gamma^2 I$ with $\gamma = 0.001$. For the prior model set $\beta=10$, $\alpha=2$, $\tau = 3$ and the true scaling parameter $\lambda_\ast = 0.1$.

For the SGD method we have chosen a step size $\beta_k = 0.001k^{-1}.$ The learned parameter moves fast into direction of the true $\lambda_\ast$, and oscillates around this value, where the variance reduces with the iterations, as seen in Figure \ref{fig:SGD_nonlinear}.

Finally, Figure \ref{fig:Tik_min_nonlinear_par} highlights again the importance and improvements of choosing the right regularization parameters.

\begin{figure}[!htb]
	\begin{center}
	\includegraphics[width=0.6\textwidth]{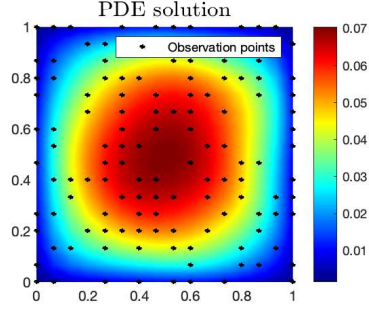}
    \caption{Reference PDE solution for Darcy flow of the underlying unknown parameter $u^\dagger$ and the corresponding randomized observation points $x_1,\dots,x_K\in  {X}$.}\label{fig:obs_points2}
    \end{center}
\end{figure}

\begin{figure}[!htb]
	%\begin{center}
	\includegraphics[width=0.9\textwidth]{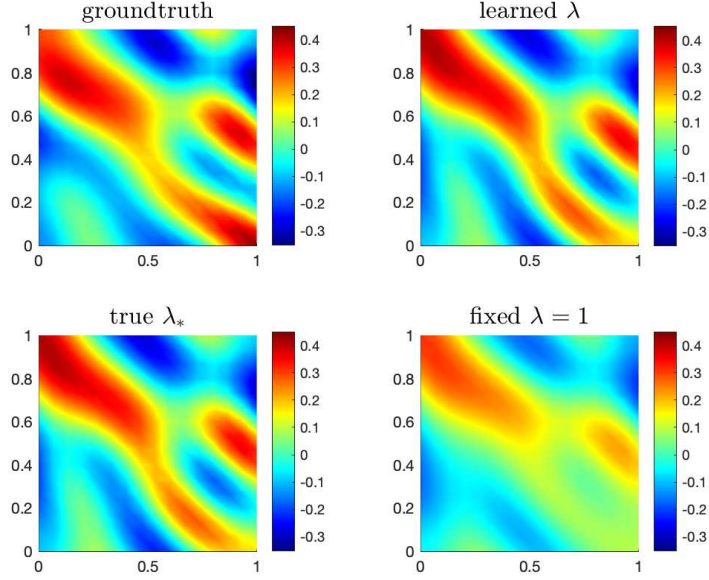}
    \caption{Comparison of different Tikhonov solutions for choices of the regularization parameter $\lambda$. The learned Tikhonov regularized solution corresponds to the resulting one of the SGD method Algorithm \ref{alg:aBSGD} for Darcy flow.}
    \label{fig:Tik_min_nonlinear_par}
    %\end{center}
\end{figure}

\begin{figure}[!htb]
	\begin{center}
	\includegraphics[width=0.5\textwidth]{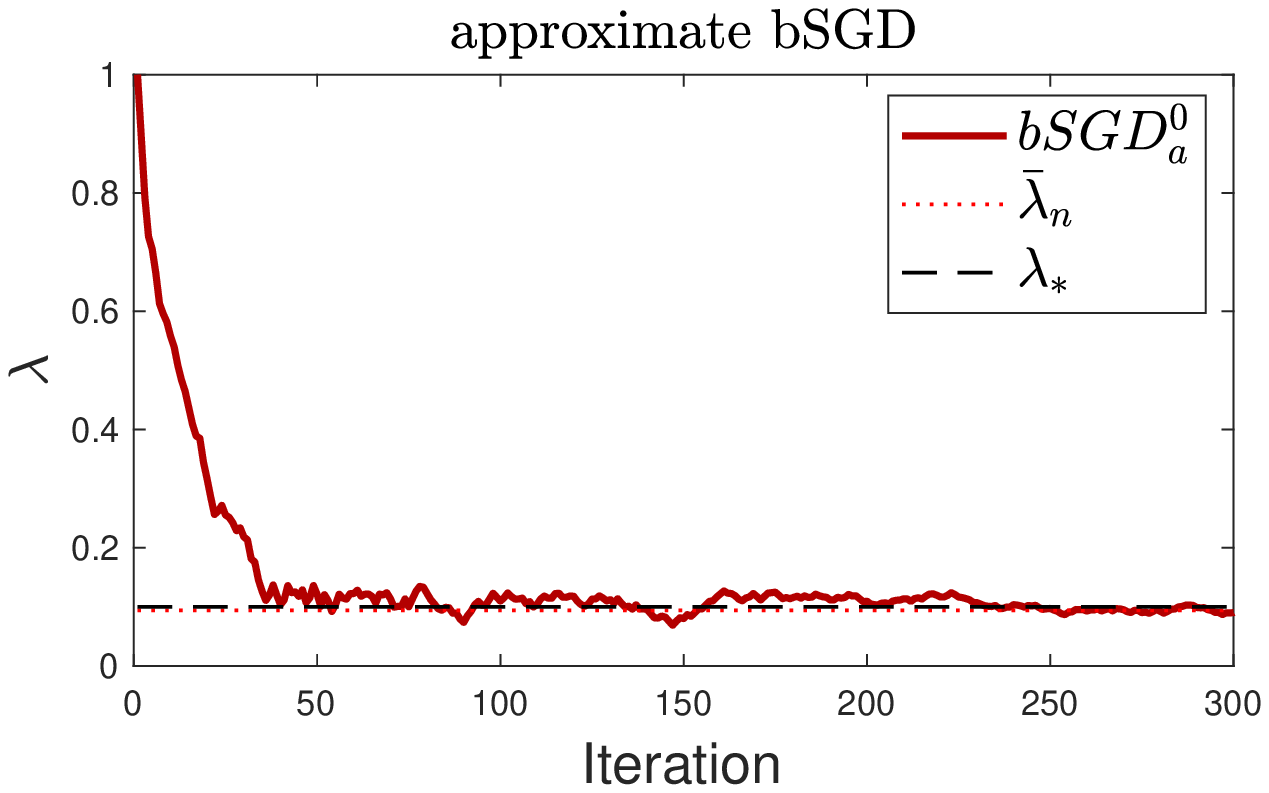}
    \caption{Learned regularization parameter $\lambda_k$, for Darcy flow, resulting from the approximate bilevel SGD method Algorithm \ref{alg:aBSGD} with fixed derivative accuracy $h=h_0$ and the corresponding mean over the last $50$ iterations $\bar\lambda_n$. {We obtain an error $|\lambda_\ast-\bar\lambda_n|^2= 3.3640\mathrm{e}{-05}$.}}
    \label{fig:SGD_nonlinear}
    \end{center}
\end{figure}

\subsection{Nonlinear example: Eikonal equation}
\label{ssec:eik}
We also seek to test our theory on the eikonal equation, which is concerned with wave propagation. Given a slowness or inverse velocity function 
$s(x) \in C^0(\bar{{ {X}}})=:\mathcal{U}$, characterizing the medium, and a source
location $x_0 \in { {X}}$, the 
forward eikonal equation is to solve for travel time 
$T(x) \in C^0(\bar{{ {X}}})=:\mathcal{V}$ satisfying  
\begin{equation}\label{eq:eikonal}
\begin{cases}
\qquad |\nabla T(x)| &=\ s(x), \quad x \in { {X}}\setminus \{x_0\} , \\
 \  \qquad T(x_0) &=\ 0, \\
\nabla T(x) \cdot \nu(x) & \geq \ 0, \quad \quad \ x \in \partial {{X}}.
\end{cases}
\end{equation}
The forward solution $T(x)$ represents the shortest travel time from $x_0$ to a point in the domain ${ {X}}$. The Soner boundary condition imposes that the wave propagates along the unit outward normal  $\nu(x)$ on the boundary of the domain.
The model equation \eqref{eq:eikonal} is of the form \eqref{eq:IP_PDE_constr} with an additional constrain arising from the Soner boundary condition.

The inverse problem for \eqref{eq:eikonal}  is to determine the speed function $s=\exp(u)$ from measurements of the shortest travel time $T(x)$. The data is assumed to take the form
\begin{equation}
\label{eq:func}
 y= \mathcal{O}(T) + \eta,
\end{equation}
where $\mathcal O:\mathcal{V}\to \R^K$ denotes the linear observation map, which takes again measurements at $K=125$ randomly chosen grid points in $ {X}$, i.e.~$\mathcal O(p) = (T(x_1),\dots,T(x_K))^\top$, for $T\in\mathcal{Z}$, $x_1,\dots,x_K\in  {X}$. The observed points can be seen in Figure~\ref{fig:obs_points3}. The measurements noise is again denoted by $\eta\in\mathcal N(0,\Gamma)$, for $\Gamma\in\R^{K\times K}$ symmetric and positive definite. {Again we formulate the inverse problem through
\begin{equation}
\label{eq:ip_eikonal}
y^\dagger =  \mathcal{G}(u^\dagger)+\eta,
\end{equation}
with $ \mathcal{G} = \mathcal O\circ G$}, where $G:\mathcal{U} \to\mathcal{V}$ denotes the solution operator of \eqref{eq:df}. As before we will assume our unknown $u^{\dagger}$ is distributed according to a mean-zero Gaussian with covariance structure \eqref{eq:cov}. For this numerical example we set $\beta = 1,\ \tau = 0.1,\ \alpha = 2$ and $\lambda_\ast = 0.1$. We truncate the KL expansion such that the unknown parameters $\xi\in\R^d$ with $d=25$. For the eikonal equation we take a similar approach to Section \ref{ssec:nlinear}, that is we use the SGD described through Algorithm \ref{alg:aBSGD}. For the SGD method we have chosen an adaptive step size $${\beta_k = \min\left(0.002,\frac{\lambda_0}{|\partial_\lambda f(\lambda_k,Z^{(k)})|}\right) k^{-1}}.$$ {Here, the chosen step size $\beta_k$ provides a bound on the maximal moved step in each SGD step, i.e.
\begin{equation}
|\beta_k\cdot\partial_\lambda f(\lambda_k,Z^{(k)})|\le \lambda_0/k.
\end{equation}
This helps to avoid instability arising through the high variance of the stochastic gradient, but the step size will be mainly of order $0.002/k$. However, from theoretical side it is not clear whether assumption of \eqref{eq:RM_cond} is still satisfied. Therefore, we will also show the resulting $\sum_{k=1}^n\beta_k$ and the realisation of the stochastic gradient $\partial_\lambda f(\lambda_k,Z^{(k)})$ in Figure~\ref{fig:step_size_realisation}.}

Our setting for the parameter choices of our prior and for the bilevel-optimization problem remain the same. To discretize \eqref{eq:eikonal} on a uniform mesh with $256$ grid points we use a fast marching method, described  by the work of Sethian \cite{EDS11,JAS99}.

As we observe the numerical experiments, Figure \ref{fig:Tik_min_eik_par} highlights that using the learned $\lambda_n$ provides recoveries almost identical to that of using the true $\lambda_{*}$. For both cases we see an improvement over the case $\lambda=1$ which is what we expected and have seen throughout our experiments. This is verified through Figure \ref{fig:SGD_eik} where we see oscillations of the learned $\lambda_k$ around the true $\lambda_{\ast}$, until approximately 100 iterations where it starts to become stable.  {Finally from Figure \ref{fig:step_size_realisation} we see that the summation of our choice $\beta_k$  diverges, but not as quickly as the summation of the deterministic step size $0.002/k$ does, which is the implication of the introduced adaptive upper bound based on the size of the stochastic gradient $\partial_\lambda f(\lambda_k,Z^{(k)})$. Figure \ref{fig:step_size_realisation} also shows the histrogram of the stochastic gradient and its rare realized large values.}

\begin{figure}[!htb]
	\begin{center}
	\includegraphics[width=0.6\textwidth]{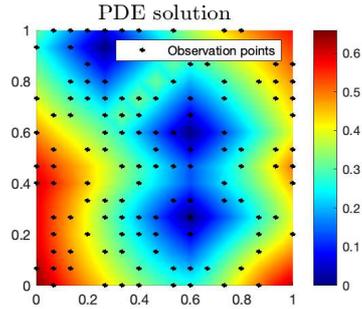}
    \caption{Reference PDE solution for the eikonal equation of the underlying unknown parameter $u^\dagger$, and the corresponding randomized observation points $x_1,\dots,x_K\in  {X}$.}\label{fig:obs_points3}
    \end{center}
\end{figure} 

\begin{figure}[!htb]
	%\begin{center}
	\includegraphics[width=0.9\textwidth]{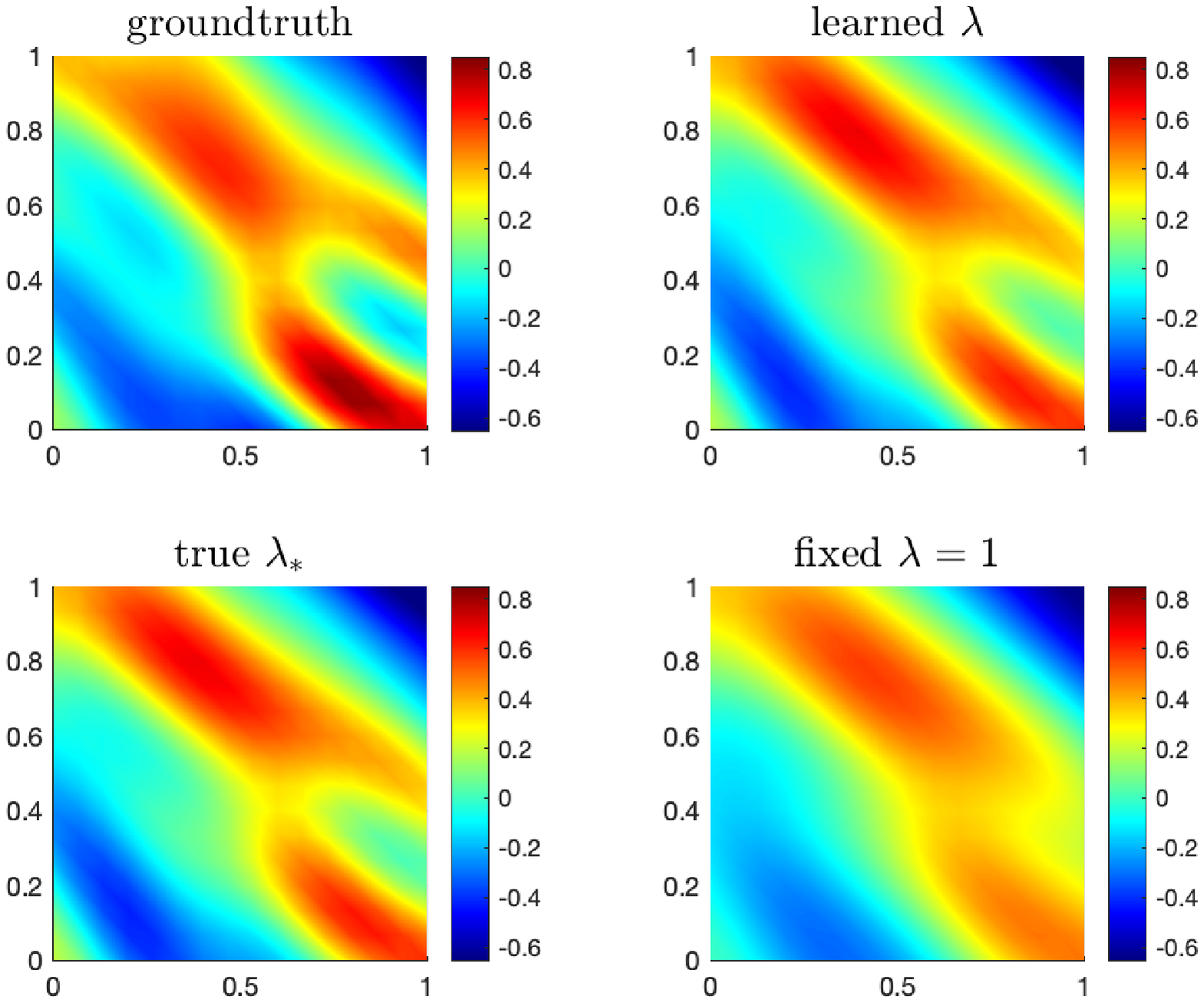}
    \caption{Comparison of different Tikhonov solutions for choices of the regularization parameter $\lambda$. The learned Tikhonov regularized solution corresponds to the resulting one of the SGD method Algorithm \ref{alg:aBSGD} for the eikonal equation.}
    \label{fig:Tik_min_eik_par}
    %\end{center}
\end{figure}

\begin{figure}[!htb]
	%\begin{center}
	\includegraphics[width=0.5\textwidth]{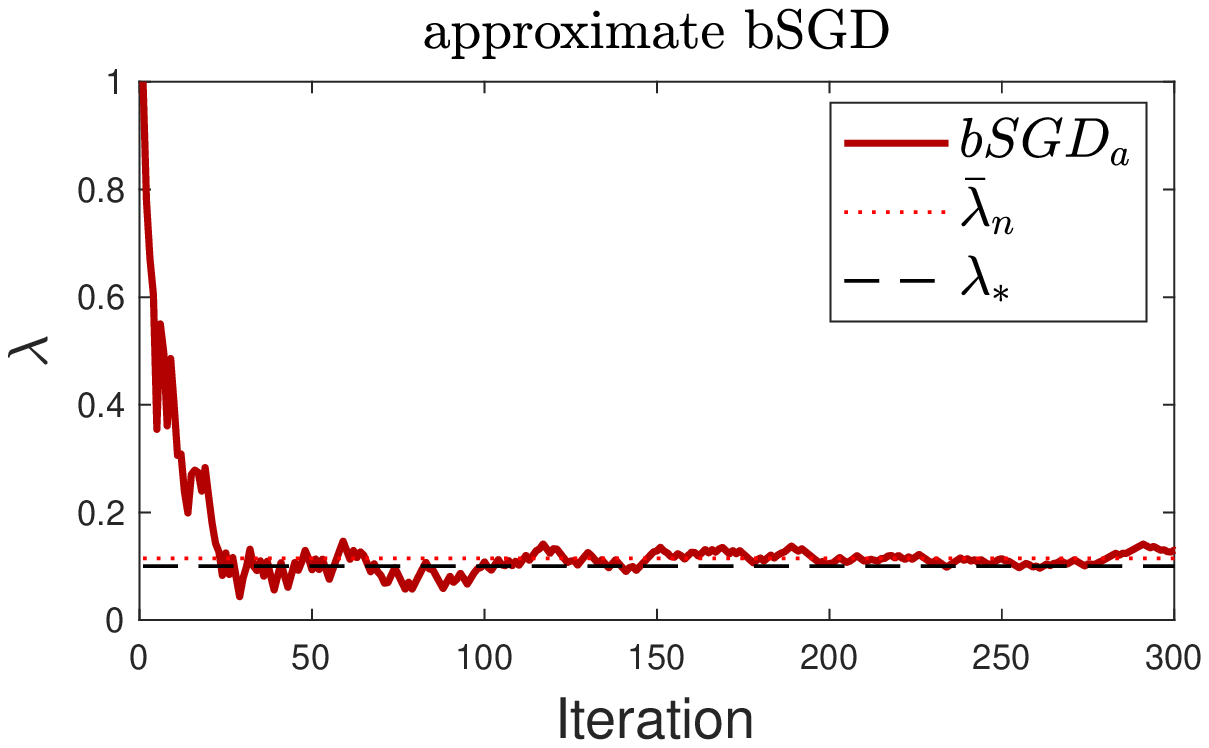}
    \caption{Learned regularization parameter $\lambda_k$, for the eikonal equation, resulting from the approximate bilevel SGD method Algorithm \ref{alg:aBSGD} with fixed derivative accuracy $h=h_0$ and the corresponding mean over the last $50$ iterations $\bar\lambda_n$. {We obtain an error $|\lambda_\ast-\bar\lambda_n|^2= 1.9360\mathrm{e}{-05}$.}}
    \label{fig:SGD_eik}
    %\end{center}
\end{figure} 

\begin{figure}[!htb]
	\begin{subfigure}[b]{0.49\textwidth}
\includegraphics[width=1.1\textwidth]{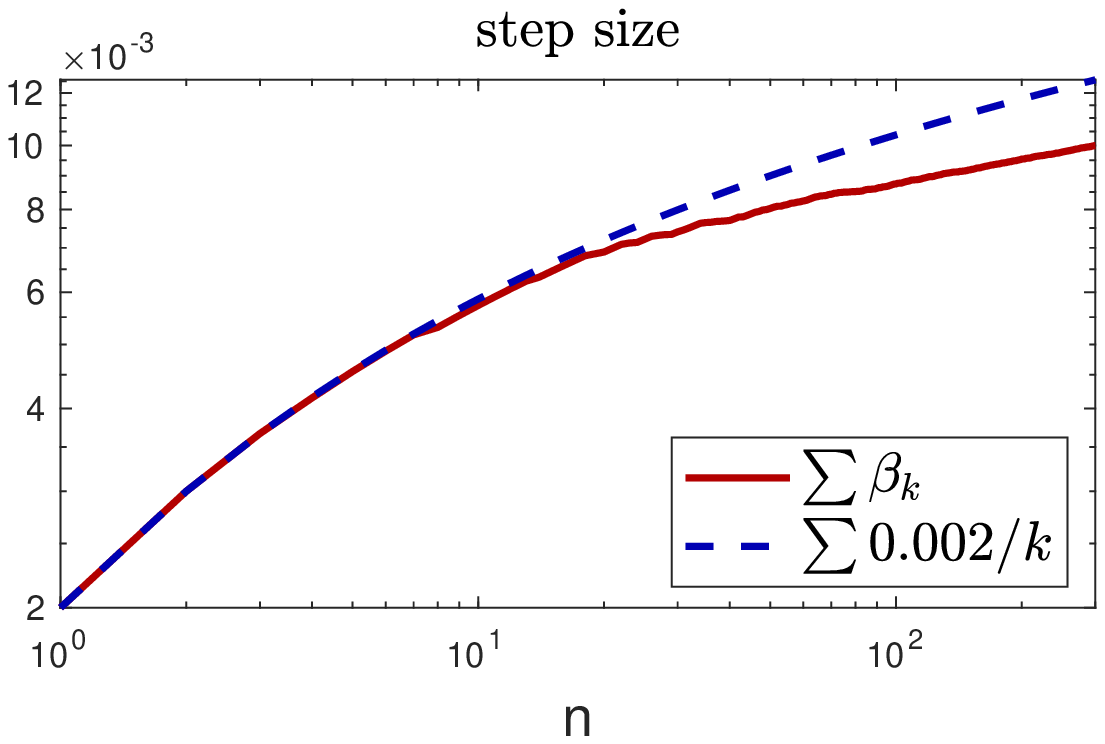}
	\end{subfigure}
	\begin{subfigure}[b]{0.49\textwidth}
	\includegraphics[width=1.1\textwidth]{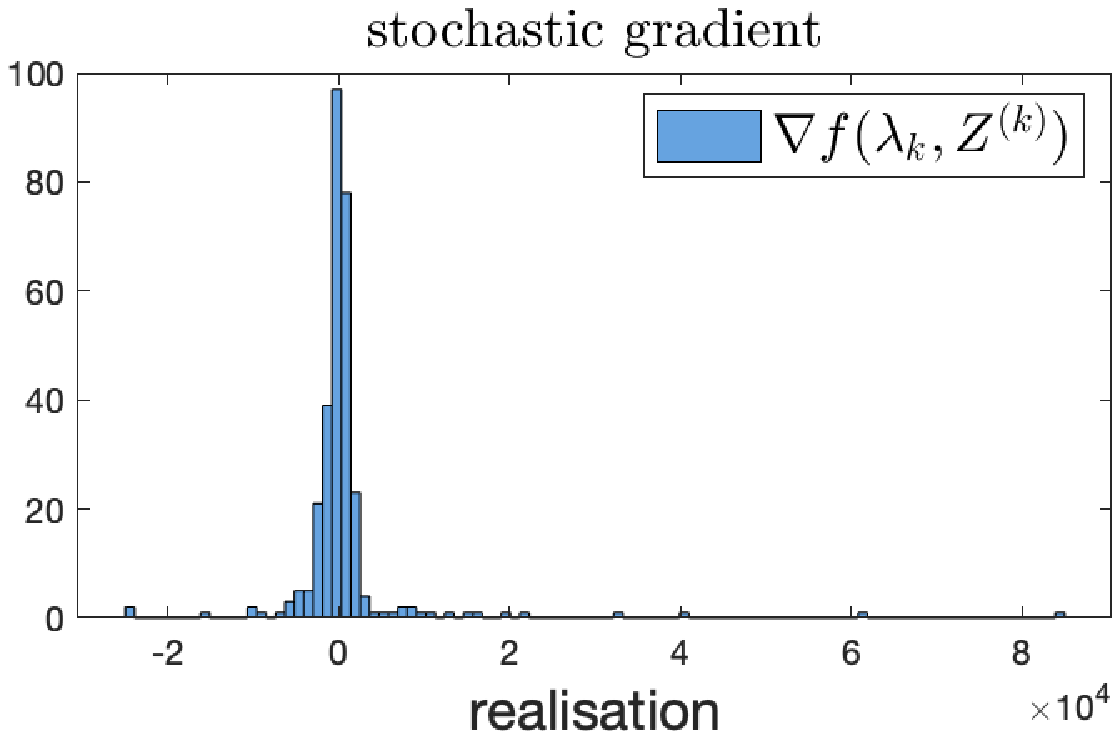}
	\end{subfigure}
    \caption{Summation of the realized adaptive step size (left) and the realized stochastic gradient $\nabla f(\lambda_k,Z^{(k)})$ (right) resulting from the online recovery, Algorithm~\ref{alg:SGD} for the eikonal equation.}\label{fig:step_size_realisation}
\end{figure}

\subsection{Signal denoising example}

{We now consider implementing our methods on image denoising, which is discussed in Section \ref{sec:intro} and subsection \ref{ssec:image}.  We are interested in denoising a 1D compound Poisson process of the form
\begin{equation}
\label{eq:pp}
u_t = \sum^{N_t}_{i=1}X_i,
\end{equation}
where $(N_t)_{t\in[0,T]}$ is a Poisson process, with rate $r>0$ and $(X_i)_{i=1}^{N_t}$ are i.i.d. random variables representing the jump size.
Here, we have chosen $X_1\sim\mathcal N(0,1)$. We consider the task of recovering a perturbed signal of the form \eqref{eq:pp} through Tikhonov regularization with different choices of regularization parameter $\lambda$. In particular, the observed signal $u=(u_{t_1},\ldots,u_{t_d})^{\top} \in \R^d$ is perturbed by white noise
\begin{equation}
\label{eq:observed_signal}
y_{t_i} = u_{t_i} + \eta_{t_i},
\end{equation}
where $t_i\in\{1/d\cdot T,2/d\cdot T,\dots,T\}$ and $\eta_{t_i}\sim\mathcal N(0,\sigma^2)$ are i.i.d. random variables, and the Tikhonov estimate corresponding to the lower level problem of \eqref{eq:imaging_bilevel} for given regularization parameter $\lambda>0$ is defined by
\begin{equation}\label{eq:Tik_signal}
u_\lambda(y) = (\Gamma^{-1} + \lambda L^{-1})^{-1}\Gamma^{-1}y,
\end{equation}
with given regularization matrix $L\in\R^{d\times d}$ and $y=(y_1,\dots,y_d)^\top\in\R^d$. We assume to have access to training data $(u^{(j)},y^{(j)})_{j=1}^n$ of \eqref{eq:observed_signal} and choose the regularization parameter $\widehat\lambda$ according to Algorithm~\ref{alg:SGD}. Further, we compare the resulting estimate of the signal
$$y_{\mathrm obs} = u^\dagger+\eta,$$
to fixed choices of $\lambda\in\{0.01,\ 0.00001\}$ and to the best possible choice $\lambda_\ast=\argmin_\lambda\ \|u_\lambda(y_{\mathrm obs})-u^\dagger\|^2$.\\
For the experiment we set the rate of jumps $r =  10$ and consider the signal observed up to time $T=1$ at $d=1000$ observation points. For Algorithm~\ref{alg:SGD}, we use a training data set of size $n=500$, we set an initial value $\lambda_0 = 0.001$ and step size $\beta_k = 0.001 k^{-1}$. The Tikhonov solution \eqref{eq:Tik_signal} has been computed with a second-order regularization matrix $L=\Delta^{-1}$. As we can see from our results the value of $\lambda=0.001$ oversmoothens the estimate in comparison with $\lambda=0.00001$. This is shown in Figure \ref{fig:denoising_recon_fixed}. However comparing fixed $\lambda$ with the learned $\lambda$ in Figure \ref{fig:denoising_recon_adaptive} we see an improvement, closer to the best possible $\lambda$, which is verified further through Table \ref{table:1}, where we can see the MSE over the time intervall. Both Figure \ref{fig:denoising_recon_fixed} and Figure \ref{fig:denoising_recon_adaptive} show on the right hand side the pointwise squared error over time. 

\begin{figure}[!htb]
	%\begin{center}
	\includegraphics[width=0.95\textwidth]{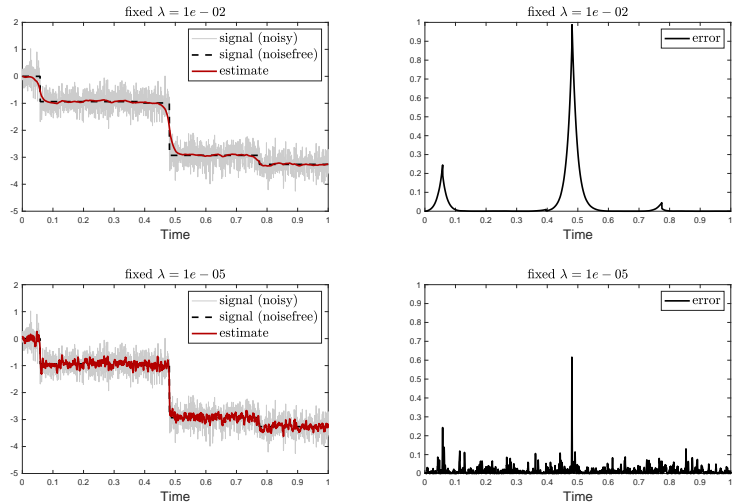}
    \caption{Comparison of different Tikhonov solutions for fixed choices of the regularization parameter $\lambda$ for the signal denoising example.}
    \label{fig:denoising_recon_fixed}
    %\end{center}
\end{figure} 

\begin{figure}[!htb]
	%\begin{center}
	\includegraphics[width=0.95\textwidth]{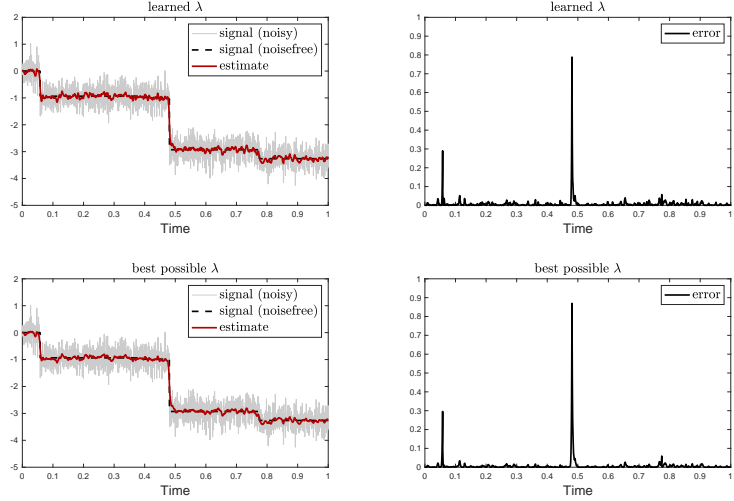}
    \caption{Comparison of the learned to best possible Tikhonov solutions for choices of the regularization parameter $\lambda$. The learned Tikhonov regularized solution corresponds to the resulting one of the SGD method Algorithm \ref{alg:SGD} for the signal denoising example.}
    \label{fig:denoising_recon_adaptive}
    %\end{center}
\end{figure}

\begin{table}
\begin{center}
\begin{tabular}{|l|c|c|c|c|}
\hline
$\boldsymbol{\lambda}$ & $1e-02$ & $1e-05$ & $\overline\lambda_n$ & $\lambda_\ast$ \\
\hline
\textbf{error} & $0.0378$ & $0.0134$ & $0.0077$ & $0.0073$ \\
\hline
\end{tabular}
\bigskip
\caption{MSE over time of the reconstruction for different choices of the regularization parameter for signal denoising example.}
\label{table:1}
\end{center}
\end{table}

\section{Conclusion}
\label{sec:conc}

In this work we have provided new insights into the theory of bilevel learning for inverse problems. In particular our focus was on deriving statistical consistency results with respect to the data limit of the regularization parameter $\lambda$. This was considered for both the offline and online representations of the bilevel problem. For the online version we used and motivated stochastic gradient descent as the choice of optimizer, as it is well known to reduce the computational time required compared to other methodologies. To test our theory we ran numerical experiments on various PDEs which not only verified the theory, but clarified that adapting the regularization parameter $\lambda$ outperforms that of a fixed value. 
Our results in this article provide numerous directions for future, both practically and analytically.}
\begin{itemize}
\item One direction is to consider a fully Bayesian approach, or understanding, to bilevel learning. In the context of statistical inverse problems, this could be related to treating $\lambda$ as a hyperparameter of the underlying unknown. This is referred to as hierarchical learning \cite{PRS07} which aims to improve the overall accuracy of the reconstruction \cite{ABPS14,DIS17}.
\item  Another potential direction is to understand statistical consistency from other choices of regularization. Answering this for other penalty terms, such as $L_1$, total variation and adversarial \cite{LOB18} (based on neural networks), is of importance and interest in numerous applications \cite{AMOS19}. A potential first step in this direction would be to consider the well-known elastic-net regularization \cite{EHN96}, which combines both $L_1$ and Tikhonov regularization. Of course to consider this one would need to modify the assumptions on convexity. 
\item Finally one could propose using alternative optimizers, which provide a lower computational cost. A natural choice would be derivative-free optimization \cite{ER20}. One potential optimizer could be ensemble Kalman inversion \cite{CSY20}, a recent derivative-free methodology, which is of particular interest to the authors. In particular as EKI has been used in hierarchical settings \cite{CIRS18,NKC18}, the reduction in cost could be combined with the hierarchical motivation discussed above. 
\end{itemize}

\section*{Acknowledgements} 
NKC acknowledges a Singapore Ministry of Education Academic Research Funds Tier 2 grant [MOE2016-T2-2-135] and KAUST baseline funding. SW is grateful to the DFG RTG1953 "Statistical Modeling of Complex Systems and Processes" for funding of this research. The research of XTT is supported by the National University of Singapore grant R-146-000-292-114.

\appendix

\section{Proofs of offline consistency analysis}
\subsection{General framework}

We start with the proof for the general framework:
\begin{proof}[Proof of Proposition \ref{prop:cond_conv}]
To simplify the mathematical notation, we use  $z$ to denote the data couple $(u,y)$, and use $f$ to denote the data loss function
\[
f(\lambda,z)=\calL_\calU(u_\lambda(y),u).
\]

When $\lambdahat_n\in \mathcal D$, we apply the fundamental theorem of calculus on $\partial_{\lambda} \Fhat_n$, and find
\[
\partial_{\lambda}\Fhat_n(\lambda_\ast)=\partial_{\lambda}\Fhat_n(\lambdahat)+\int^1_0 \partial^2_{\lambda} \Fhat_n(s\lambda_\ast+(1-s)\lambdahat)(\lambda_\ast-\lambdahat_n)ds=A_F(\lambdahat_n-\lambda_\ast),
\]
where
\[
A_F:=\int^1_0 \partial^2_{\lambda} \Fhat_n ((1-s)\lambdahat_n+s\lambda_\ast)ds\succeq c_o I. 
\]
Note that 
\begin{align*}
{\bf 0}=\partial_{\lambda}F(\lambda_\ast)&=\partial_{\lambda}F(\lambda_\ast)-\partial_{\lambda} \Fhat_n(\lambda_\ast)+\partial_{\lambda} \Fhat_n(\lambda_\ast)\\
&=A_F(\lambdahat_n-\lambda_\ast)+\partial_{\lambda} F(\lambda_\ast)-\partial_{\lambda} \Fhat_n(\lambda_\ast).
\end{align*}
We can reorganize this as
\[
-\left(\partial_{\lambda} \Fhat_n(\lambda_\ast)-\partial_{\lambda} F(\lambda_\ast)\right)=A_F(\lambdahat_n-\lambda_\ast),
\]
As a consequence, we now have a formula for the point estimation error $\lambda_\ast-\lambdahat_n$.
\begin{equation}
\label{tmp:w-w}
\|\lambda_\ast-\lambdahat\|=\left\|A_F^{-1}\left(\partial_{\lambda} \Fhat_n(\lambda_\ast)-\partial_{\lambda} F(\lambda_\ast)\right)\right\|\leq c_0^{-1}\left\|\partial_{\lambda} \Fhat_n(\lambda_\ast)-\partial_{\lambda} F(\lambda_\ast)\right\|.
\end{equation}
Note that by using $\partial_\lambda\E f(\lambda,Z) = \E\partial_\lambda f(\lambda,Z)$, see \cite[Theorem 12.5]{RS17},
\[
\partial_{\lambda} \Fhat_n (\lambda_\ast)-\partial_{\lambda} F(\lambda_\ast)=\frac1n\sum_{i=1}^n \partial_{\lambda}f(\lambda_\ast,z_i)-\E \partial_{\lambda} f(\lambda_\ast,Z).
\]
So 
\[
\E \|\partial_{\lambda} \Fhat_n (\lambda_\ast)-\partial_{\lambda} F(\lambda_\ast)\|^2=\frac1n\text{tr}(\text{Var}(\partial_{\lambda}f(\lambda_\ast,Z))).
\]
And our second claim follows by Cauchy-Schwarz
\[
\E 1_{\calA_n} \|\partial_{\lambda} \Fhat_n (\lambda_\ast)-\partial_{\lambda} F(\lambda_\ast)\|\leq \sqrt{\E \|\partial_{\lambda} \Fhat_n (\lambda_\ast)-\partial_{\lambda} F(\lambda_\ast)\|^2}.
\]
\end{proof}

\subsection{Formulas for the Linear inverse problem}

Next, we apply Proposition \ref{prop:cond_conv} to linear the inverse problem {under Assumption~\ref{aspt:subGaussian}}. 

{
The solution of the Tikhonov regularized optimization problem (without assuming any distribution on $y_i$ and $\xi_i$ respectivelly) in the linear setting can be written as 
$$ u_\lambda(y_i) = (A^\top \Gamma^{-1} A + \lambda C_0^{-1})^{-1} A^\top \Gamma^{-1} y_i,$$ and we consider the difference
$$  \tilde u_i -u_\lambda(y_i)\overset{d}{=} - (A^\top\Gamma^{-1}A +\lambda C^{-1}_0)^{-1} A^\top\Gamma^{-1}(Au_i-\Gamma^{1/2}\xi_i)+u_i.$$
}

Denote $D:=C_0^{1/2}A^\top\Gamma^{-1/2},\Omega_0=C_0^{-1}, v_i=\Omega^{1/2}_0u_i$, and $\xi_i=\Gamma^{-1/2}(A u_i-y_i)\sim \mathcal{N}(0,I)$
and note that 
\begin{align*}
%&-(A^\top\Gamma^{-1}A +\lambda C^{-1}_0)^{-1} A^\top\Gamma^{-1}y_i+u_i\\
&(A^\top\Gamma^{-1}A +\lambda C^{-1}_0)^{-1} A^\top\Gamma^{-1}(Au_i-\Gamma^{1/2}\xi_i)+u_i\\
 &=(A^\top\Gamma^{-1}A +\lambda C^{-1}_0)^{-1} (\lambda C_0^{-1}u_i+A^\top \Gamma^{-1/2}\xi_i)\\
&=C_0^{1/2}(C_0^{1/2}A^\top\Gamma^{-1}A C_0^{1/2}+\lambda I)^{-1} C_0^{1/2} (\lambda C_0^{-1}u_i+A^\top \Gamma^{-1/2}\xi_i)\\
&=C_0^{1/2}(DD^\top +\lambda I)^{-1}(\lambda v_i+D \xi_i).
\end{align*}
Therefore we define
\[
Q_\lambda=(DD^\top +\lambda I)^{-1},
\]
and the data loss can be written as
\begin{align*}
f(\lambda,z)&=\Tr(  Q_\lambda C_0Q_\lambda (\lambda v+D\xi)(\lambda v+D\xi)^\top )\\
&=\Tr( Q_\lambda C_0Q_\lambda(\lambda^2v v^\top+2\lambda D\xi v^\top+D\xi\xi^\top D^\top)).
\end{align*}
We further define the following quantities
\[
P_1=Q_\lambda C_0Q_\lambda,\quad P_2=\frac{\partial P_1}{\partial \lambda}=-(Q_\lambda^2C_0Q_\lambda+Q_\lambda C_0Q^2_\lambda),
\]
\[
P_3=\frac{\partial P_2}{\partial \lambda}=2(Q_\lambda^3C_0Q_\lambda+Q_\lambda^2 C_0 Q^2_\lambda+Q_\lambda C_0Q^3_\lambda),
\]
\[
P_4=\frac{\partial P_3}{\partial \lambda}=-6(Q_\lambda^4C_0Q_\lambda+Q_\lambda^3 C_0 Q_\lambda^2+Q_\lambda^2 C_0 Q_\lambda^3+Q_\lambda C_0Q^4_\lambda).
\]
Note that $\|Q_\lambda\|\leq \lambda^{-1}$, so we have
\[
|\Tr(Q_\lambda^k C_0 Q_\lambda^j)|=|\Tr(C_0 Q_\lambda^{j+k})|\leq \frac{1}{\lambda^{j+k}}\Tr(C_0),
\]
\[
\|Q_\lambda^k C_0 Q_\lambda^j\|\leq \|Q_\lambda\|^{j+k}\|C_0\|\leq \frac{1}{\lambda^{j+k}}\|C_0\|,
\]
\[
\|Q_\lambda^k C_0 Q_\lambda^j\|_F\leq \|Q_\lambda^k\| \|C_0 Q_\lambda^j\|_F\leq  \|Q_\lambda\|^{j+k}\|C_0\|_F\leq \frac{1}{\lambda^{j+k}}\|C_0\|_F.
\]
In conclusion, for function $T$ being $T(A)=|\Tr(A)|$ or $T(A)=\|A\|$ or $T(A)=\|A\|_F$, we all have
\[
T(P_k)\leq (\tfrac{2}{\lambda})^{k+1} T(C_0).
\]
In particular, under Assumption \ref{aspt:DI}, $T(P_k)$ will be bounded by constants independent of the dimension. 

Using these notations, we have
\begin{align*}
\partial_\lambda f(\lambda,z)=\Tr\bigg( & P_2(\lambda^2 v v^\top+2\lambda D\xi v^\top+D\xi\xi^\top D^\top)+P_1(2\lambda vv^\top+2 D\xi v^\top) \bigg).
\end{align*}
\begin{align*}
\partial^2_\lambda f(\lambda,z)=\Tr\bigg( & P_3(\lambda^2 vv^\top+2\lambda D\xi v^\top+D\xi\xi^\top D^\top) +4P_2(\lambda vv^\top+D\xi v^\top)+2P_1 vv^\top \bigg). 
%=2\Tr\bigg(& (\Omega_0^2-4\lambda Q_\lambda^3\Omega_0+Q_\lambda^2)vv^\top+2(3\lambda Q_\lambda^4\Omega^2_0-2Q_\lambda^3\Omega_0)D\xi v^\top+3Q_\lambda^4 \Omega_0^2D\xi\xi^\top D^\top\bigg)
\end{align*}

\begin{align*}
\partial^3_\lambda f(\lambda,z)=\Tr\bigg( & P_4(\lambda^2 vv^\top+2\lambda D\xi v^\top+D\xi\xi^\top D^\top) +6P_3(\lambda vv^\top+D\xi v^\top)+6P_2 vv^\top \bigg). 
%=2\Tr\bigg(& (\Omega_0^2-4\lambda Q_\lambda^3\Omega_0+Q_\lambda^2)vv^\top+2(3\lambda Q_\lambda^4\Omega^2_0-2Q_\lambda^3\Omega_0)D\xi v^\top+3Q_\lambda^4 \Omega_0^2D\xi\xi^\top D^\top\bigg)
\end{align*}

\subsection{Pointwise consistency analysis}
To apply Proposition \ref{prop:cond_conv}, it is necessary to show the gradient of $\Fhat_n(\lambda)$ is a good approximation of $\partial F(\lambda)$ at $\lambda=\lambda_\ast$ with high probability. This is actually true for general $\lambda$.

To show this, we start by showing the sample covariance are consistent.
\begin{lem}\label{lem:HW}
{Let $z_i = (z_1^i,\dots,z_d^{i})^\top\in\R^d$ with $z_l^i$ i.i.d. random variables in $i=1,\dots,n$ as well as in $l=1,\dots,d$ with $\E z_1^1= 0$, $\E|z_1^1|^2 = 1$ and $$\sup_{p\ge1}\ p^{-1/2}\E[|z_1^1|^p]^{1/p}\le C_v,$$
and $\xi = (\xi_1^i,\dots, \xi_K^i)^\top\in\R^K$ with $\xi_l^i$ i.i.d. random variables in $i=1,\dots,n$ as well as in $l=1,\dots,K$ with $\E\xi_1^1 = 0$, $\E|\xi_1^1|^2 = 1$ and $$\sup_{p\ge1}\ p^{-1/2}\E[|\xi_1^1|^p]^{1/p}\le C_\xi.$$
}
For
\[
C_n = \frac1n \sum\limits_{i=1}^n z_i z_i^\top, \quad B_n= \frac1n\sum\limits_{i=1}^n z_i \xi_i^\top,
\]
the following holds
\[
\P(|\Tr(\Sigma C_n)-\Tr(\Sigma)|>t)\leq 2\exp\left(-cn\min \left( \frac{t^2}{\|\Sigma\|^2_F},\frac{t}{\|\Sigma\|}\right)\right),
\]
\[
\P(|\Tr(\Sigma B_n)|>t)\leq 2\exp\left(-cn\min \left( \frac{t^2}{\|\Sigma\|^2_F},\frac{t}{\|\Sigma\|}\right)\right).
\]
Moreover, there is a universal constant $C$ such that 
\[
\E |\Tr(\Sigma C_n)|^2\leq 2\Tr(\Sigma)^2 + C \left(\frac{\|\Sigma\|^3_F}{n^{3/2}}+\frac{\|\Sigma\|^3}{n^{3}}\right). 
\]
\[
\E |\Tr(\Sigma B_n)|^2\leq C \left(\frac{\|\Sigma\|^3_F}{n^{3/2}}+\frac{\|\Sigma\|^3}{n^{3}}\right). 
\]
\end{lem}
\begin{proof}
 Note 
\[
\Tr(\Sigma z_iz_i^\top)=z_i^\top\Sigma z_i.
\]
We define the block-diagonal matrix $D_\Sigma\in \R^{nd\times nd}$ which consists of $n$ blocks of $\Sigma$, and $Z=[z_1;z_2;\cdots;z_n]\in \R^{nd}$. Note that
\[
\Tr(\Sigma C_n)=Z^\top(\tfrac1nD_\Sigma) Z.
\]
By the Hanson--Wright inequality \cite[Theorem 1.1]{RV13}, we obtain for some constants $c_0$ and $K$,
\[
\P(|\Tr(\Sigma C_n)-\Tr(\Sigma)|>t)\leq 2\exp\bigg(-c_0\min \bigg(\frac{t^2}{K^4\|\tfrac1nD_{\Sigma}\|_F^2},\frac{t}{K^2\|\tfrac1nD_{\Sigma}\|}\bigg)\bigg).
\]
Note that 
\[
\|\tfrac1nD_{\Sigma}\|_F^2=\frac1{n^2}\|D_{\Sigma}\|^2_F=\frac1n\|\Sigma\|^2_F,
\]
\[
\|\tfrac1nD_{\Sigma}\|=\frac1n\|D_\Sigma\|=\frac1n \|\Sigma\|. 
\]
So the first assertion is proved.
For the second claim we first note that 
\[
\Tr(\Sigma \xi_iz_i^\top)=z_i^\top\Sigma \xi_i=\begin{bmatrix} z_i^\top, \xi_i^\top\end{bmatrix}
Q\begin{bmatrix} z_i\\ \xi_i\end{bmatrix},\quad Q=\begin{bmatrix} 0 &\Sigma \\ 0 &0\end{bmatrix}\in \R^{(d+d_y)\times (d+d_y)}.
\]
Consider then a block-diagonal matrix $D_Q\in \R^{n(d+y)\times n(d+d_y)}$ which consists of $n$ blocks of $Q$, and $Z=[z_1; \xi_1;z_2;\xi_2;\cdots;z_n;\xi_n]\in \R^{n(d+d_y)}$. Then we can verify that
\[
\Tr(\Sigma B_n)=Z^\top(\tfrac1nD_Q) Z.
\]
By the Hanson--Wright inequality \cite[Theorem 1.1]{RV13}, we have 
\[
\P(|\Tr(\Sigma B_n)|>t)\leq 2\exp\bigg(-c\min \bigg(\frac{t^2}{K^4\|\tfrac1nD_{Q}\|_F^2},\frac{t}{K^2\|\tfrac1nD_{Q}\|}\bigg)\bigg).
\]
Again we note that 
\[
\|\tfrac1nD_{Q}\|_F^2=\frac1{n^2}\|D_{Q}\|^2_F=\frac1n\|\Sigma\|^2_F,
\]
\[
\|\tfrac1nD_{Q}\|=\frac1n\|D_Q\|=\frac1n \|Q\|. 
\]
and finally end up with 
\[
\P(|\Tr(\Sigma B_n)|>t)\leq 2\exp\left(-cn\min \bigg(\frac{t^2}{K^4\|\Sigma\|_F^2},\frac{t}{K^2\|\Sigma\|}\bigg)\right).
\]
For the bounds of second moments, let $T=\Tr(\Sigma C_n)-\Tr(\Sigma)$, then 
\[
\E |\Tr(\Sigma C_n)|^2\leq 2|\Tr(\Sigma)|^2+2 \E T^2.   
\]
Then note that by the probability bound,
\begin{align*}
\E T^2&\leq \int^\infty_0 t^2 \P( |T|>t)dt\\
&\leq 2\int t^2 \exp\left(-cn\min \left( \frac{t^2}{\|\Sigma\|^2_F},\frac{t}{\|\Sigma\|}\right)\right)dt\\
&\leq 2\int^\infty_0 t^2 \exp\left(-cn \frac{t^2}{\|\Sigma\|^2_F}\right)dt+2\int^\infty_0t^2\exp\left(-cn\frac{t}{\|\Sigma\|}\right)dt.\\
\end{align*}
We let $s=\frac{t\sqrt{cn}}{\|\Sigma\|_F}$ and find
\[
\int^\infty_0 t^2 \exp\left(-cn \frac{t^2}{\|\Sigma\|^2_F}\right)dt=
\frac{\|\Sigma\|^3_F}{(cn)^{\frac{3}{2}}}\int^\infty_0 s^2 \exp\left(-s^2\right)ds. 
\]
We then let $s=cn\frac{t}{\|\Sigma\|}$ and find 
\[
\int^\infty_0t^2\exp\left(-cn\frac{t}{\|\Sigma\|}\right)dt=
\frac{\|\Sigma\|^3}{(cn)^3}\int^\infty_0s^2\exp\left(-s\right)ds. 
\]
So there is a universal constant $C$ such that 
\[
\E T^2\leq \frac C 2\left(\frac{\|\Sigma\|^3_F}{n^{3/2}}+\frac{\|\Sigma\|^3}{n^{3}}\right). 
\]
The bound for $\E |\Tr (\Sigma B_n)| $ follows identically. 

\end{proof}

By the previous result we obtain the following convergence results.

\begin{lem}\label{cor:conv_derivatives}
The empirical loss function $\Fhat_n$ is $C^3$ in $\lambda$, and for any $\lambda\in(\lambda_l,\lambda_r)$, there exists constants $C,\ c >0$ such that for all $\varepsilon>0$
\[
\P(|\partial_{\lambda}\Fhat_n(\lambda) - {(1-\lambda/\lambda_\ast)\Tr(P_2DD^{\top} )}|>\varepsilon)\le C \exp(-nc\min\{\varepsilon,\varepsilon^2\}),
\]
and
\[
\P(|\partial_{\lambda}^2\Fhat_n(\lambda) - \Tr\left((\frac{\lambda^2}{\lambda_\ast}-\lambda)P_3+(\frac{4\lambda}{\lambda_\ast}-3\lambda) P_2+\frac{2}{\lambda_\ast} P_1\right)|>\varepsilon)\le C \exp(-cn\min\{\varepsilon,\varepsilon^2\}).
\]
%where we have defined $D:=A^\top \Gamma^{-1/2}$ and $Q_\lambda=(DD^\top +\lambda \Omega_0)^{-1}$.
 Under Assumption \ref{aspt:DI}, both $C$ and $c$ are independent of ambient dimension $d$. 
\end{lem}
\begin{proof}
Since $\Fhat_n(\lambda)=\frac1n\sum_{i=1}^n f(\lambda,\zeta_i)$, if we let
\[
C_v= \frac1n\sum_{i=1}^n v_i v_i^\top,\quad  B=\frac1n\sum_{i=1}^n \xi_i v_i^\top,\quad C_{\xi}=\frac1n\sum_{i=1}^n \xi_i\xi_i^\top,
\]
then
\begin{align}
\label{tmp:F1}
\partial_\lambda \Fhat_n(\lambda)	&=\Tr\bigg(P_2(\lambda^2 C_v+2\lambda DB+DC_{\xi} D^\top)+P_1(2\lambda C_v+2 DB) \bigg)\\
&=\Tr\bigg((P_2\lambda^2+2\lambda P_1) C_v+(2P_1+2\lambda P_2) DB+D^\top P_2DC_{\xi}\bigg)
%						&=-2\lambda^2\Tr(Q_\lambda^3 \Omega_0C_v)-4\lambda\Tr(Q_\lambda^3\Omega_0DB)-2\Tr(D^\top Q_\lambda^3 \Omega_0D C_\xi) + 2\lambda\Tr(Q_\lambda^2C_v) \\
%\notag
%&\quad+2\Tr(Q_\lambda^2DB).
%\label{tmp:F1}						&=-2\lambda^2/\lambda_\ast\Tr(\Omega_0^{1/2}Q_\lambda^3 \Omega_0^{3/2}C'_v)-4\lambda/\sqrt{\lambda_\ast}\Tr(\Omega_0^{1/2} Q_\lambda^3\Omega_0DB')-2\Tr(D^\top Q_\lambda^3\Omega_0 DC_\xi)\\ 
%&+ 2\lambda/\lambda_\ast\Tr(\Omega_0^{1/2}Q_\lambda^2 \Omega_0^{1/2}C'_v) 
%\notag
%+2\Tr(\Omega_0^{1/2}Q_\lambda^2DB').											
\end{align}
We note that
\begin{align*}
\E \partial_{\lambda}\Fhat_n(\lambda)&=\Tr\bigg(P_2(\lambda^2/\lambda_\ast I +D D^\top)+2\lambda P_1/\lambda_\ast \bigg)\\
&=\Tr\bigg(-(Q_\lambda^2C_0 Q_\lambda+Q_\lambda C_0 Q_\lambda^2)(\tfrac{\lambda^2}{\lambda_\ast} I +D D^\top)+2\tfrac{\lambda}{\lambda_\ast} Q_\lambda C_0Q_\lambda \bigg)\\
&=\Tr\bigg( (1-\tfrac{\lambda}{\lambda_\ast})P_2D D^\top \bigg).\\
\end{align*}

Moreover, in \eqref{tmp:F1}, $\partial_\lambda \Fhat_n$ can be written as sum of $\Tr(\Sigma_1 C_v), \Tr(\Sigma_2 B)$ and $\Tr(\Sigma_3 C_\xi)$ for certain matrices $\Sigma$ such as
\[
\Sigma_1=(P_2\lambda^2+2\lambda P_1),\quad \Sigma_2=(2P_1+2\lambda P_2) D,\quad \Sigma_3=D^\top P_2 D.
\]
Note that for any random variables $A_k$
\begin{equation*}
\P\left(|\sum\limits_{k=1}^m(A_k-\E A_k)|>\varepsilon\right)\le\sum\limits_{k=1}^m\P(|A_k-\E A_k|>\varepsilon/m).
\end{equation*}
Therefore we can apply Lemma \ref{lem:HW} at each trace term, and bound its probability of deviating from its mean. Therefore, we can find constants $C_1,c$ such that 
\[
\P(|\partial_{\lambda}\Fhat_n(\lambda)-(1-\lambda/\lambda_\ast)\Tr(P_2 DD^\top)|>\varepsilon)\leq C_1 \exp\left(-cn\min(\varepsilon^2,\varepsilon)\right).
\]

If Assumption \ref{aspt:DI} is assumed, note that 
\[
\|Q_\lambda\|\le \frac1\lambda\|C_0\|\le\frac 1{\lambda_l}\|C_0\|,
\]
 for all $\lambda\geq \lambda_l$  respectively. So using norm inequalities $\|AB\|\leq \|A\|\|B\|$ and $\|AB\|_F\leq \|A\|\|B\|_F$, we can verify that
 all $P_i, i=1,\ldots, 4$ have dimension independent operator norm and Frobenius norms.  Therefore all $\Sigma_i,i=1,\ldots,3$ have dimension independent operator norm and Frobenius norms.  So $C_1$ and $c$ are dimension independent. 

For the second claim,
\begin{align*}
\partial^2_\lambda f(\lambda,z)=\Tr\bigg( & P_3(\lambda^2 vv^\top+2\lambda D\xi v^\top+D\xi\xi^\top D^\top) +4P_2(\lambda vv^\top+D\xi v^\top)+2P_1 vv^\top \bigg). 
%=2\Tr\bigg(& (\Omega_0^2-4\lambda Q_\lambda^3\Omega_0+Q_\lambda^2)vv^\top+2(3\lambda Q_\lambda^4\Omega^2_0-2Q_\lambda^3\Omega_0)D\xi v^\top+3Q_\lambda^4 \Omega_0^2D\xi\xi^\top D^\top\bigg)
\end{align*}
 we find
\begin{equation}
\label{tmp:F2}
\partial^2_\lambda \Fhat_n(\lambda)	= \Tr\bigg( (\lambda^2 P_3+4\lambda P_2+2 P_1)C_v+(2\lambda P_3+4P_2) D B+ D^\top P_3D C_\xi\bigg).
\end{equation}
Therefore, 
\[
\E\partial_{\lambda}^2\Fhat_n(\lambda) = \Tr\left((\lambda^2 P_3+4\lambda P_2+2 P_1)/\lambda_\ast+DD^\top P_3\right).
\]
The deviation probability can also be obtained by analyzing matrices
\[
 \Sigma_1'=(\lambda^2 P_3+4\lambda P_2+2P_1),\quad \Sigma_2'=(2\lambda P_3+4P_2) D,\quad \Sigma_3'=D^\top P_3D.
\]
Note that 
\[
\Tr(Q_\lambda^{-1} P_3)=\Tr(2 Q_\lambda^2 C_0Q_\lambda+Q_\lambda^2 C_0Q_\lambda+Q_\lambda C_0Q^2_\lambda+2Q_\lambda C_0Q^2_\lambda)=-3P_2.
\]
\[
\Tr(\lambda P_2 +2P_1)= \Tr ((Q_\lambda^{-1}-\lambda I)Q^2_\lambda C_0 Q_\lambda+Q_\lambda C_0 Q^2_\lambda (Q_\lambda^{-1}-\lambda I))=-\Tr(DD^\top P_2 ). 
\]
So the average can also be written as 
\begin{align*}
\E\partial_{\lambda}^2\Fhat_n(\lambda) &= \Tr\left((\lambda(\lambda I+DD^\top) P_3+4\lambda P_2+2 P_1)/\lambda_\ast+(1-\lambda/\lambda_\ast)DD^\top P_3\right)\\
&= \Tr\left((-3\lambda P_2+4\lambda P_2+2 P_1)/\lambda_\ast+(1-\lambda/\lambda_\ast)DD^\top P_3\right)\\
&=\Tr\left(-DD^\top P_2/\lambda_\ast+(1-\lambda/\lambda_\ast)DD^\top P_3\right)\\
&=\Tr\left(DD^\top((1-\lambda/\lambda_\ast) P_3-P_2/\lambda_\ast)\right).
\end{align*}

Likewise, we can obtain
\begin{equation*}
\partial^3_\lambda \Fhat_n(\lambda)	=  \Tr\bigg( (\lambda^2 P_4+6\lambda P_3+6P_2)C_v+(2\lambda P_4+6P_2) D B+ D^\top P_4D C_\xi\bigg),
\end{equation*}
and
\begin{equation}
\label{tmp:F3}
\partial^3_\lambda \Fhat_n(\lambda)	=  \Tr\bigg( (\lambda^2 P_4+6\lambda P_3+6P_2)/\lambda_\ast+D^\top P_4D \bigg).
\end{equation}

\end{proof}

\begin{remark}\label{rem:app_convexity}
It is worthwhile to note that 
\[
\partial_{\lambda}^2F(\lambda)=\E\partial_{\lambda}^2\Fhat_n(\lambda)= \Tr\left(DD^\top((1-\lambda/\lambda_\ast) P_3-P_2/\lambda_\ast)\right),
\]
is not always positive, and it can be negative if $\lambda$ is very large. In other words, $F$ is not convex on the real line. Therefore, it is necessary to introduce a local parameter domain where $F$ is convex inside. 
\end{remark}

\subsection{Consistency analysis within an interval}
To apply Proposition \ref{prop:cond_conv}, it is also necessary to show the $\Fhat_n(\lambda)$ is strongly convex in a local region/interval. This can be done using a chaining argument in probability theory. 

First, we show that the empirical loss function has bounded derivatives with high probability. 
\begin{lem}
\label{lem:Lipschitz}
There exists an $L>0$ and $c$ that only depend on $\Tr(C_0), \|C_0\|_F$ and $\|D\|$ such that the following holds true
\[
\P\left(\max_{\lambda_l\leq\lambda\leq \lambda_u} |\partial^k_\lambda \Fhat_n (\lambda)| > L,k=1,2,3\right)\leq 6\exp(- nc).
\] 
Under Assumption \ref{aspt:DI}, $L$ and $c$ are dimension independent constants. 
\end{lem}
\begin{proof}
Recall that $\|Q_\lambda\|\le \frac1{\lambda_l}$.
From \eqref{tmp:F1}, \eqref{tmp:F2} and \eqref{tmp:F3}, and Lemma \ref{lem:chaining}, we have 
\[
\mathbb{P}(\max_{\lambda_l\leq\lambda\leq \lambda_u} |\partial^k_\lambda \Fhat_n (\lambda)-\E  \Fhat_n (\lambda)|>t)\leq 2\exp\left(-cn\min \left( \frac{t^2}{\|\Sigma_k\|^2_F},\frac{t}{\|\Sigma_k\|}\right)\right),
\]
for each $k=1,2,3$. Here each $\Sigma_k$ consists of matrices of form $P_j S$ or $SP_j$ where $j=1,2,3,4$ and $S=I, D$ or $DD^\top$. Then because 
\[
\|P_jS\|\leq \|P_j\|\|S\|\leq \frac{\|C_0\|\|S\|}{\lambda_l^{j+1}},\quad \|P_jS\|_F\leq \frac{\|C_0\|_F\|S\|}{\lambda_l^{j+1}}.
\]
So we see that $c$ can depend on  $\|C_0\|\leq \Tr(C_0), \|C_0\|_F$ and $\|D\|$. Meanwhile, $\E \partial^k_\lambda \Fhat_n$ is a linear sum of some $\Tr(P_j)$. While 
\[
|\Tr(P_j)|\leq (\tfrac{2}{\lambda_l})^{j+1} \Tr(C_0).
\]
So $L$ can also be taken as a constant that depends only on $\Tr(C_0), \|C_0\|_F$ and $\|D\|$. This concludes our proof.
\end{proof}

Next, we show that if a function is bounded at each fixed point with high probability, it is likely to be bounded on a fixed interval if it is Lipschitz. 
\begin{lem}
\label{lem:chaining}
 Let $f_n(\lambda)$ be function of $\lambda$ and the following is true for some interval $I=[\lambda_l,\lambda_u]$
\[
\P(f_n(\lambda)>a)\leq  C\exp(-n c_a  )\quad \forall  \lambda_l\leq\lambda\leq \lambda_u.
\]
Then 
\[
\P\left({\max_{\lambda\in I }}f_n(\lambda)>2a, \max_{\lambda\in I}|\partial f_n(\lambda)|\leq M\right)\leq  a^{-1}|\lambda_u-\lambda_l|MC\exp(-nc_a).
\]
Let $f_n(\lambda)$ be function of $\lambda$ and the following is true for some interval $I$
\[
\P(f_n(\lambda)<a)\leq  \exp(-n c_a  )\quad \forall  \lambda_l\leq\lambda\leq \lambda_u.
\]
Then 
\[
\P\left(\min_{\lambda\in I }f_n(\lambda)<a/2, \max_{\lambda\in I}|\partial f_n(\lambda)|\leq M\right)\leq  2a^{-1}|\lambda_u-\lambda_l|MC\exp(-nc_a).
\]
\end{lem}
\begin{proof}
Pick $\lambda_i=\lambda_l+\frac{2a}{|M|} i$ for $i=0,\ldots, \lfloor \frac{|\lambda_u-\lambda_l|M}{2a}\rfloor$. Then $\lambda_l\leq \lambda_i\leq \lambda_u$, and for any $\lambda_l\leq \lambda\leq \lambda_u$, $|\lambda-\lambda_i|\leq \frac{a}{M}$ for some $\lambda_i$. Not that if $|\partial f_n(\lambda)|\leq M$, and $f_n(\lambda_i)\leq a$, for all $i$, then for any $\lambda_l\leq \lambda\leq \lambda_u$,
\[
f_n(\lambda)\leq f_n(\lambda_i)+(\lambda_i-\lambda)\partial_\lambda f_n(\lambda)\leq a+\frac{a}{M}M=2a.
\]
Consequentially, by union bound 
\begin{align*}
\P\bigg(\min_{\lambda\in I }f_n(\lambda)>2a, \max_{\lambda\in I}|\partial f_n(\lambda)|\leq M\bigg)&\leq \P\bigg(f_n(\lambda_i)>a \text{ for some }i\bigg) \\ & \leq a^{-1}|\lambda_u-\lambda_l|MC\exp(-nc_a).
\end{align*}
The same argument can be applied to show the second claim, except that we choose $\lambda_i=c+\frac{a}{|M|}$.
\end{proof}

The next lemma indicates that the loss function is strongly convex within $\calD$ with high probability. 

\begin{lem}\label{lem:hessian}
Assume that  the largest eigenvalue of $DD^\top$ is $\lambda_D$ and let $\lambda\in \mathcal D := [\frac5 6\lambda_\ast, \frac 76\lambda_\ast]$. Then for some constants $c,C>0$, 
\begin{align*}
\P(\min_{\lambda\in \mathcal D}\partial_\lambda^2 \Fhat_n(\lambda) &< H_\ast/4) 
			\le \frac{C}{\min\{H_\ast,1\}}\exp\bigg(-cn\min(H_\ast^2,H_\ast,1)\bigg),
\end{align*}
with
\[H_\ast = H_\ast(\lambda_D,\lambda_\ast) =\frac{\lambda_D^2 }{(\lambda_D+2\lambda_\ast)^2\lambda_\ast\|A^\top \Gamma^{-1}A\|}>0.
\]
\end{lem}
\begin{proof}
Denote
\[
\mathcal A=DD^\top((1-\lambda/\lambda_\ast) P_3-P_2/\lambda_\ast),
\]
and $v_i$ being the eigenvector of $DD^\top$ corresponds to eigenvalue $\lambda_i$. Note that $v$ is also the eigenvector $Q_\lambda$ with eigenvalue $(\lambda_i+\lambda)^{-1}$, then 
\[
v_i^\top \mathcal Av_i= \left(\frac{6(1-\lambda/\lambda_\ast)}{(\lambda_i+\lambda)^3} +\frac{2 }{(\lambda_i+\lambda)^2\lambda_\ast}\right) \lambda_i v_i^\top C_0 v_i. 
\]
When $\lambda\in\mathcal{D}$, if $7/6\lambda_\ast\geq \lambda>\lambda_\ast$
\[
\frac{6(1-\lambda/\lambda_\ast)}{(\lambda_i+\lambda)^3} +\frac{2 }{(\lambda_i+\lambda)^2\lambda_\ast}
\geq-\frac{1}{(\lambda_i+\lambda)^2\lambda }+\frac{2 }{(\lambda_i+\lambda)^2\lambda_\ast}\geq \frac1{(\lambda_i+\lambda)^2\lambda_\ast}.
\]
If $\lambda\leq \lambda_\ast$, the same relation also holds. 
Then note that if $v_i$ are all the eigenvectors of $DD^\top$ with eigenvalues $\lambda_i$,  while $\lambda_i$ are decreasing,
\[
\Tr \mathcal A=\sum_{i=1}^d v^\top_i \mathcal Av_i\geq  \sum_{i=1}^d \frac{\lambda_iv^\top_i C_0v_i}{(\lambda_i+2\lambda_\ast)^2\lambda_\ast}
=\frac{\lambda_1 v_1^\top C_0v_1}{(\lambda_1+2\lambda_\ast)^2\lambda_\ast}.
\]
Finally, note that 
\[
\lambda_D=\lambda_1=v^\top_1DD^\top v_1=v_1^\top C_0^{1/2} A^\top \Gamma^{-1} A C_0^{1/2} v_1\leq \|A^\top \Gamma^{-1}A\|\|C_0^{1/2} v_1\|^2.
\]
So
\[
\Tr \mathcal A\geq \frac{\lambda_D^2 }{(\lambda_D+2\lambda_\ast)^2\lambda_\ast\|A^\top \Gamma^{-1}A\|}=H_\ast.
\]

for $\lambda\in\mathcal D$ and we set $\varepsilon = H_\ast/2>0$ to apply Corollary \ref{cor:conv_derivatives}. We obtain some $C_1,c$
\begin{align*}
&C_1\exp\left(-cn\min(H_\ast^2,H_\ast)\right) \\ &\ge\P\Bigg(|\partial_\lambda^2 \Fhat_n(\lambda)-\frac{2}{\lambda_\ast}\Tr\left((3\lambda_\ast I-2\lambda I + DD^\top)DD^\top Q_\lambda^4\right)|>H_\ast/2\Bigg)\\
		&\ge \P\Bigg(\partial_\lambda^2 \Fhat_n(\lambda)<\frac{2}{\lambda_\ast}\Tr\left((3\lambda_\ast I-2\lambda I + DD^\top)DD^\top Q_\lambda^4\right)-H_\ast/2\Bigg)\\
		&\ge \P(\partial_\lambda^2 \Fhat_n(\lambda)<H_\ast/2).
\end{align*}
By Lemma \ref{lem:Lipschitz} there exists an $L>0$ and $c_1$ such that 
\[ 
\P\left(\max_{\lambda\in\mathcal D}|\partial^3_\lambda \Fhat_n(\lambda|)>L\right)\le 6\exp(-nc_1),
\]
and by Lemma \ref{lem:chaining} it holds true that

\begin{align*}
\frac{C_2}{\min(H_\ast,1)}&\exp\bigg(-cn\min(H_\ast^2,H_\ast,1)\bigg)\\
						&\ge \P\left(\min_{\lambda\in \mathcal D}\partial_\lambda^2 \Fhat_n(\lambda) < H_\ast/4,\max_{\lambda\in\mathcal D}|\partial^3_\lambda \Fhat_n(\lambda)|\le M\right),
\end{align*}
for some $C_2>0$. We define the sets $A_n:=\{\min_{\lambda\in \mathcal D}\partial_\lambda^2 \Fhat_n(\lambda) < H_\ast/4\}$ and $B_n:=\{\max_{\lambda\in\mathcal D}|\partial^3_\lambda \Fhat_n(\lambda)|\le L\}$, and we obtain
\begin{align*}
\P(\min_{\lambda\in \mathcal D}\partial_\lambda^2 \Fhat_n(\lambda) < H_\ast/4) &= \P(A_n\mid B_n)\P(B_n) + \P(A_n\mid B_n^\complement)\P(B_n^\complement)\\
			&\le \P(A_n\cap B_n) + \P(B_n^\complement)\\
%			&\le \frac2{H_\ast}M\lambda_\ast K_1\exp\bigg(-nK_2\min(\frac14H_\ast^2,\frac12H_\ast)\bigg) + 6 \exp(-nc_1).
			&\le \frac{C}{\min(H_\ast,1)}\exp\left(-\ cn\min(H_\ast^2,H_\ast,1)\right).
\end{align*}

\end{proof}

The last lemma indicates the empirical loss function is unlikely to have local minimums outside $[\frac23\lambda_\ast,\frac43 \lambda_\ast]$.
\begin{lem}\label{lem:derivative}
Assume again that  the largest eigenvalue of $DD^\top$ is $\lambda_D $. Let 
\[
 L_\ast=\frac{2\lambda_D^2}{3(\lambda_D+\lambda_u)^3\|A^\top\Gamma^{-1}A\|}.
 \]
There are constants $c, C$ such that 
\begin{align*}
\P\left(\min_{\lambda_u\ge\lambda>\tfrac43\lambda_\ast}\partial_\lambda\Fhat_n(\lambda)<L_\ast/4\right)
\le\frac{C}{\min\{L_\ast,1\}}\exp\left(-cn\min\left(L_\ast^2,L_\ast,1\right)\right),
\end{align*}
and
\begin{align*}
\P\left(\min_{\frac23\lambda_\ast\ge\lambda>\lambda_l}\partial_\lambda\Fhat_n(\lambda)>-L_\ast/4\right)
\le\frac{C}{\min\{L_\ast,1\}}\exp\left(-cn\min\left(L_\ast^2,L_\ast,1\right)\right).
\end{align*}
\end{lem}
\begin{proof}
We first note that  with $v$ being the leading eigenvector of $DD^\top$, 
\begin{align*}
-\Tr(P_2 DD^\top)\ge v^\top (Q_\lambda^2 C_0 Q_\lambda+Q_\lambda C_0 Q_\lambda^2) v&\geq
2\lambda_D\frac{v^\top C_0v}{(\lambda_D+\lambda_u)^3} \\ &\geq \frac{2\lambda_D^2}{(\lambda_D+\lambda_u)^3\|A^\top\Gamma^{-1}A\|}=:
 3L_\ast.
\end{align*}
For $\lambda>\frac43\lambda_\ast$ we have
\[
(1-\lambda/\lambda_\ast)\Tr(P_2DD^\top) \ge L_\ast.
\]

We set $\varepsilon=L_\ast/2$ to apply Lemma \ref{cor:conv_derivatives} and obtain
\begin{align*}
C\exp\left(-nc\min(L_\ast^2,L_\ast)\right)&\ge\P(|\partial_\lambda \Fhat_n(\lambda) - (1-\lambda/\lambda_\ast)\Tr(Q_\lambda^3)| > L_\ast/2)\\
			&\ge \P(\partial_\lambda \Fhat_n(\lambda)< (1-\lambda/\lambda_\ast)\Tr(Q_\lambda^3)-L_\ast/2)\\
			& \ge  \P(\partial_\lambda \Fhat_n(\lambda)< L_\ast/2).
\end{align*}
Similarly as in Lemma \ref{lem:hessian}, we use Lemma \ref{lem:Lipschitz} and Lemma \ref{lem:chaining} to obtain the first assertion by using 
\[
(1-\lambda/\lambda_\ast)\Tr(P_2 DD^\top) \le -L_\ast<0,
\]
for $\lambda<\frac23\lambda_\ast$.
\end{proof}
\subsection{Summarizing argument}
Finally, we are ready to prove Theorem \ref{thm:prob}.
\begin{proof}[Proof of Theorem \ref{thm:prob}]
Denote $\mathcal{D}=[\frac23\lambda_\ast, \frac43\lambda_\ast], $ 
\[H_\ast =\frac{\lambda_D^2 }{(\lambda_D+2\lambda_\ast)^2\lambda_\ast\|A^\top \Gamma^{-1}A\|}>0,
\]
and the events
\[
B=\{\lambda_l<\lambdahat_n<\lambda_u\},\quad \mathcal{A}_n=\{\lambdahat_n\in \mathcal D, \partial_\lambda^2 \Fhat_n(\lambda)\geq \tfrac14H_\ast \text{ for all }\lambda\in \mathcal D\}.
\]
First we decompose
\begin{align*}
\P(|\lambdahat_n-\lambda_\ast|>\varepsilon,B) = 
					&\ \P(|\lambdahat_n-\lambda_\ast|>\varepsilon,B\mid\mathcal A_n)\cdot\P(\mathcal A_n)\\
					&+\ \P(|\lambdahat_n-\lambda_\ast|>\varepsilon,B\mid\mathcal A_n^\complement)\cdot\P(\mathcal A_n^\complement)\\
					\le&\ \P(|\lambdahat_n-\lambda_\ast|>\varepsilon,B\mid\mathcal A_n)+\P(B\cap\mathcal A_n^\complement).
\end{align*}
In the last step we have used $\P(\lambdahat_n\le \lambda_u)=1$. By Proposition \ref{prop:cond_conv}
\begin{align*}
\P(|\lambdahat_n-\lambda_\ast|>\varepsilon,B\mid\mathcal A_n)
&\leq  \P\left(|\partial_\lambda \Fhat_n(\lambda_\ast)-\partial_\lambda F(\lambda_\ast)|>\frac14H_\ast\epsilon,B\right)\\
&=  \P\left(|\partial_\lambda \Fhat_n(\lambda_\ast)|>\frac14H_\ast\epsilon, B\right),
%&\leq C_1\exp(-nc_1\min\{\frac14H_\ast\epsilon,\frac1{16}H_\ast^2\epsilon^2\}). 
\end{align*}
which can be bounded by Lemmas \ref{cor:conv_derivatives} and \ref{lem:chaining}
\[
\P(|\lambdahat_n-\lambda_\ast|>\varepsilon\mid\mathcal A_n)\leq C_1\exp(-nc_1\min\{\epsilon,\epsilon^2\}),
\]
for some $C_1, c_1$.

We bound the probability $\P(\mathcal A_n^\complement)$ by
\[
\P(B\cap\mathcal A_n^\complement) \le \P(B, \lambdahat_n\notin \mathcal D) + \P(\{\partial_\lambda^2 \Fhat_n(\lambda)\ge H_\ast/4\ \text{for all}\ \lambda\in\mathcal D\}^\complement),
\]
and study both terms separately.
Note first, by Lemma \ref{lem:derivative}, for some constants $C_2,c_2>0$ the following holds  
\begin{align*}
\P(B,\lambdahat_n\notin\mathcal D) &\leq\ \P(\partial_\lambda\Fhat_n(\lambda)=0 \text{\ for some\ }\lambda\in(\lambda_l,\lambda_u)\setminus\mathcal D)\\
			&\leq\  C_2\exp\left(-c_2n\right).
			%&+\frac{L_\ast}4M\lambda_\ast C_1\exp\left(-C_2n\min\left(\frac1{16}L_\ast^2,\frac14L_\ast\right)\right)+12\exp(-nc_1).
\end{align*}
Second, by Lemma \ref{lem:hessian}, for some constants $C_3, c_3>0$ we obtain

\[
\P(\{\partial_\lambda^2 \Fhat_n(\lambda)\ge H_\ast/4\ \text{for all}\ \lambda\in\mathcal D\}^\complement)\le C_3\exp\left(-\ c_3n\right).
\]
So there exist some constants $C_\ast,c_\ast>0$ such that
\begin{align*}
\bbP(|\hat{\lambda}_n-\lambda_\ast|>\epsilon)\leq C_\ast\exp(-c_\ast n\min(\epsilon,\epsilon^2)).
\end{align*}
\end{proof}

\begin{proof}[Proof of Proposition \ref{prop:Lip}]
As for the last claim, note that 
\[
u_\lambda(y)=(A^\top\Gamma^{-1}A +\lambda C^{-1}_0)^{-1} A^\top\Gamma^{-1}y= C_0^{1/2}  Q_\lambda y,
\]
moreover by the chain rule, there is a $w$ between $\lambda$ and $\lambda_\ast$, so that
\[
\|Q_\lambda-Q_{\lambda_\ast}\|=|\lambda-\lambda_\ast| \|Q^2_{w} \|\leq \frac{|\lambda-\lambda_\ast|}{\lambda_l^2},
\]
and
\begin{align*}
\E_y\|u_\lambda(y)-u_{\lambda_\ast}(y)\|^2&= \E_y\Tr(C_0 (Q_\lambda-Q_{\lambda_\ast}) yy^T(Q_\lambda-Q_{\lambda_\ast}))\\
&= \Tr(C_0 (Q_\lambda-Q_{\lambda_\ast}) (AC_0 A^\top/\lambda_\ast+\Gamma)(Q_\lambda-Q_{\lambda_\ast}))\\
&\leq \frac{|\lambda-\lambda_\ast|^2}{\lambda_l^4}\|AC_0 A^\top/\lambda_\ast+\Gamma\|\Tr(C_0). 
\end{align*}
This concludes our proof. 
\end{proof}

\section{Online consistency analysis}
\subsection{Stochastic gradient decent}
We start by verifying Lemma \ref{lem:grad_reg}. 
\begin{proof}[Proof of Lemma \ref{lem:grad_reg}]
We apply the implicit function theorem to prove this statement. For fixed $y\in\R^K$, we define the function 
\begin{equation*}
\varphi(\lambda,u) := \nabla_u \Psi(\lambda,u,y).
\end{equation*}
Since $(\lambda,u)\mapsto \Psi(\lambda,u,y)$ is strictly convex, we have that for all $(\lambda,u)$ near $(\lambda_0,u_{\lambda_0})$ the Jacobian of $\varphi$ w.r.t. $u$ is invertible, i.e.~
\begin{equation*}
D_u\varphi(\lambda,u) = \nabla_u^2 \Psi(\lambda,u,y)>0.
\end{equation*}

Set $\bar\lambda\in\R^d$ arbitrary, then for $\bar u=u_{\bar\lambda}(y)$ it holds true that
\begin{equation*}
\varphi(\bar\lambda,\bar u) = 0,
\end{equation*}
and by the implicit function theorem there exists a open neighborhood $\calD\subset\R^d$ of $\lambda_0$  with $\bar\lambda\in\calD$ such that there exists a unique continuously differentiable function $\bar U:\calD\to\R^d$ with $\bar U(\bar\lambda) = \bar u$ and 
\begin{equation*}
\varphi(\lambda,\bar U(\lambda)) = 0,
\end{equation*}
for all $\lambda\in \Lambda$, i.e.~ $\bar U$ maps all $\lambda\in \Lambda$ to the corresponding regularized solution $\bar U(\lambda) = u_\lambda(y)$. Further, the partial derivatives of $\bar U$ are given by
\begin{equation*}
\frac{\partial \bar U}{\partial \lambda_i}(\lambda) = -\left[D_u\varphi(\lambda,\bar U(\lambda))\right]^{-1}\left[\frac{\partial\varphi}{\partial \lambda_i}(\lambda,\bar U(\lambda))\right].
\end{equation*}
Since the choice of $\bar\lambda\in\R^d$ is arbitrary, it follows that $\lambda\mapsto u_\lambda(y)$ is continuously differentiable with derivative given by
\begin{equation*}
\nabla_\lambda u_\lambda(y) = -\left(\nabla_{u}^2 \left[\Psi(\lambda,u_\lambda(y),y)\right]\right)^{-1}\nabla_{\lambda u}^2\left[\Psi(\lambda,u_\lambda(y),y)\right].
\end{equation*}
The computation of $\nabla_\lambda f$ can be obtained by the chain rule.

\end{proof}

\subsection{General consistency analysis framework}
\begin{proof}[Proof of Proposition \ref{prop:gensgd}]
 
Note that 
$$ \Delta_{k+1} = \chi(\lambda_k -\beta_k\widetilde\partial_\lambda f(\lambda_k,Z_k))-\lambda_\ast,$$
and apply Lemma~\ref{lem:convball}
\begin{align*}
\|\Delta_{k+1}\|^2 = \|\chi\left(\lambda_k -\beta_k\widetilde\partial_\lambda f(\lambda_k,Z_k)\right)-\lambda_\ast\|^2&\le \|\lambda_k -\beta_k\widetilde\partial_\lambda f(\lambda_k,Z_k)-\lambda_\ast\|^2\\ &= \|\Delta_k -\beta_k\widetilde\partial_\lambda f(\lambda_k,Z_k)\|^2\\
&= \|\Delta_k -\beta_k\partial_\lambda F(\lambda_k,Z_k)-\beta_k\delta_k-\beta_k \xi_k\|^2,
\end{align*}
where
\[
\delta_k=\E_k\widetilde\partial f(\lambda_k,Z)-\partial F(\lambda_k),\quad\xi_k=\widetilde\partial_\lambda f(\lambda_k,Z_k)-\E_k\widetilde\partial_\lambda f(\lambda_k,Z),
\]
is the bias and noise in the stochastic gradient, we denote the expectation conditioned on information available at step $k$ as $\E_k$ and define the first exit time of $\mathcal D$ by with $\tau=\inf\{k\ge0\mid \lambda_k\in \mathcal D\}$. Next, we note that 
\[
\E_k \|\partial f(\lambda_k,Z_k)\|^2=\|\partial_\lambda F(\lambda_k)+\delta_k\|^2+\E_k\|\xi_k\|^2.
\] 
So if $\tau\geq k$, 
\begin{align*}
\E_k\|\Delta_{k+1}\|^2&\leq \|\Delta_k\|^2-2\beta_k\Delta_k^T(\partial_\lambda F(\lambda_k)+\delta_k)+\beta_k^2\|\partial_\lambda F(\lambda_k)+\delta_k\|^2+\E_k\|\xi_k\|^2\\
&\leq \|\Delta_k\|^2-2\beta_k\Delta_k^T\partial_\lambda F(\lambda_k)+2\beta_k\|\Delta_k\|\|\delta_k\|+\beta_k^2(a+b\|\Delta_k\|^2)\\
&\leq \|\Delta_k\|^2-2c\beta_k \|\Delta_k\|^2+\frac12c\beta_k\|\Delta_k\|^2+\frac{2}{c}\beta_k\|\delta_k\|^2+\beta_k^2a+b\beta_k^2\|\Delta_k\|^2\\
&\leq (1-1.5c\beta_k+b\beta_k^2)\|\Delta_k\|^2+(a\beta_k+2\alpha_k/c) \beta_k. 
\end{align*}
Since $\beta_k<c/2b$, we have 
\[
\E_k 1_{\tau\geq k+1}\|\Delta_{k+1}\|^2\leq
\E_k 1_{\tau\geq k}\|\Delta_{k+1}\|^2\leq
1_{\tau\geq k}(1-c\beta_k)\|\Delta_{k}\|^2+(a\beta_k+2\alpha_k/c)\beta_k.
\]
Let $Q_k=1_{\tau_k\geq k}\|\Delta_k\|^2$, then we just derived that 
\[
\E Q_{k+1}\leq (1-c\beta_k)\E Q_k +(a\beta_k+2\alpha_k/c)\beta_k.
\]
Therefore by Gronwall's inequality
\begin{equation}
\label{temp:Qn}
\begin{split}
\E Q_n\leq a\sum_{k=1}^n \left(\prod_{j=k+1}^n(1-c\beta_j) \beta_k^2\right)&+\frac{2}{c}\sum_{k=1}^n \left(\prod_{j=k+1}^n(1-c\beta_j) \beta_k\alpha_k\right)\\
&+\exp\left(-c\sum_{j=1}^n\beta_j\right)\E Q_0.
\end{split}
\end{equation}
Next we look at the 2nd term of \eqref{temp:Qn}. Note that when $\alpha_k\leq \alpha_0$, then 
\begin{align*}
\frac{2}{c}\sum_{k=1}^{n} \prod_{j=k+1}^n(1-c\beta_j)\beta_k\alpha_k
&\leq \frac{\alpha_{0}}{c^2}\sum_{k=1}^{n} \prod_{j=k+1}^n(1-c\beta_j)c\beta_k \\
&\leq \frac{\alpha_{0}}{c^2}\sum_{k=1}^{n} \left(\prod_{j=k+1}^n(1-c\beta_j)-\prod_{j=k}^n(1-c\beta_j)\right)\leq \frac{\alpha_{0}}{c^2}.
\end{align*}
In this case, \eqref{temp:Qn} becomes 
\[
\E Q_n\leq a\sum_{k=1}^n \left(\prod_{j=k+1}^n(1-c\beta_j) \beta_k^2\right)+\frac{\alpha_0}{c^2}
+\exp\left(-c\sum_{j=1}^n\beta_j\right)\E Q_0.
\]
And if $\alpha_k\leq D\beta_k$, then \eqref{temp:Qn} can be simplified as 
\[
\E Q_n\leq (a+D/c)\sum_{k=1}^n \left(\prod_{j=k+1}^n(1-c\beta_j) \beta_k^2\right)+\exp\left(-c\sum_{j=1}^n\beta_j\right)\E Q_0.
\]
In both cases, to show our claim, we just need to show 
\[
\sum_{k=1}^n \left(\prod_{j=k+1}^n(1-c\beta_j) \beta_k^2\right)\leq 2C_n,
\quad \exp\left(-c\sum_{j=1}^n\beta_j\right)\leq C_n. 
\]

Let $k_0$ be the minimizer of 
\[
k_0=\arg\min_{k\leq n}\max\{\prod_{j=k+1}^n(1-c\beta_j), a\beta_k/c\}.
\]
Then note that,
\begin{align*}
\sum_{k=1}^{k_0} \prod_{j=k+1}^n(1-c\beta_j)\beta_k^2
&\leq \sum_{k=1}^{k_0} \prod_{j=k_0+1}^n(1-c\beta_j)\beta_k^2\\
&\leq \prod_{j=k_0+1}^n(1-c\beta_j)\sum_{k=1}^\infty \beta_k^2
\leq C_n. 
\end{align*}
also 
\begin{align*}
\sum_{k=k_0+1}^{n} \prod_{j=k+1}^n(1-c\beta_j)\beta_k^2
&\leq \frac{1}{c}\beta_{k_0}\sum_{k=1}^{k_0} \prod_{j=k+1}^n(1-c\beta_j)c\beta_k \\
&\leq \frac{1}{c}\beta_{k_0}\sum_{k=1}^{k_0} \left(\prod_{j=k+1}^n(1-c\beta_j)-\prod_{j=k}^n(1-c\beta_j)\right) \\
&\leq \frac{1}{c}\beta_{k_0}=C_n.
\end{align*}
The sum of the previous two inequalities leads to 
\[
\sum_{k=1}^n \left(\prod_{j=k+1}^n(1-c\beta_j) \beta_k^2\right)\leq 2C_n.
\]
Finally
\[
\exp(-c\sum_{j=1}^n\beta_j)\E Q_0\leq 
\exp(-c\sum_{j=k_0+1}^n\beta_j)\E Q_0\leq C_n. 
\]
To see that $C_n$ converges to zero, simply let 
\[
k_{n}=\max_{k}\left\{\prod_{j=1}^k(1-c\beta_j)>\sqrt{\prod_{j=1}^n(1-c\beta_j)}\right\}.
\]
Because $\prod_{j=1}^n(1-c\beta_j)$ decays to zero when $n\to \infty$, so $k_n$ will increases to $\infty$, and $\beta_{k_n}$ will decay to zero. 
Meanwhile,
\[
C_n\leq \min\left\{\prod_{j=k+1}^n(1-c\beta_j),\beta_{k_n}\right\}
\leq \min\left\{\sqrt{\prod_{j=1}^n(1-c\beta_j)},\beta_{k_n}\right\},
\]
which will decay to zero when $n\to \infty$.
\end{proof}

\subsection{Application to linear inverse problems}\label{app:proof_linear_online}
\begin{proof}[Proof of Theorem \ref{thm:linSGD}]

We will set $\calD=\Lambda=[\lambda_l,\lambda_u]$. Then because $\chi$ always bring $\lambda_k$ back into $\calD$, the event $\calA$ always happen.

Recall that 
\begin{align*}
\partial_\lambda f(\lambda,z)=\Tr\bigg( & P_2(\lambda^2 v v^\top+2\lambda D\xi v^\top+D\xi\xi^\top D^\top)+P_1(2\lambda vv^\top+2 D\xi v^\top) \bigg).
\end{align*}
\begin{align*}
\partial^2_\lambda f(\lambda,z)=\Tr\bigg( & P_3(\lambda^2 vv^\top+2\lambda D\xi v^\top+D\xi\xi^\top D^\top) +4P_2(\lambda vv^\top+D\xi v^\top)+2P_1 vv^\top \bigg). 
%=2\Tr\bigg(& (\Omega_0^2-4\lambda Q_\lambda^3\Omega_0+Q_\lambda^2)vv^\top+2(3\lambda Q_\lambda^4\Omega^2_0-2Q_\lambda^3\Omega_0)D\xi v^\top+3Q_\lambda^4 \Omega_0^2D\xi\xi^\top D^\top\bigg)
\end{align*}
\begin{equation}
\label{tmp:f3}
\partial^3_\lambda f(\lambda,z)=\Tr\bigg(  P_4(\lambda^2 vv^\top+2\lambda D\xi v^\top+D\xi\xi^\top D^\top) +6P_3(\lambda vv^\top+D\xi v^\top)+6P_2 vv^\top \bigg). 
\end{equation}

We have seen in the proof of Lemma \ref{lem:derivative}, that
\begin{equation*}
- \partial_\lambda F(\lambda)=(\lambda/\lambda_\ast -1)\Tr(P_2DD^\top).
\end{equation*}
Multiplication with $(\lambda-\lambda_\ast)$ gives
\begin{equation*}
-(\lambda-\lambda_\ast)\partial_\lambda F(\lambda) = (\lambda-\lambda_\ast)^2\Tr(P_2D^TD/\lambda_\ast).
\end{equation*}
Note that if $v$ is the eigenvector of $DD^\top$ with maximum eigenvalue $\lambda_D$
\[
\Tr(P_2DD^\top/\lambda_\ast)\geq v^\top P_2 DD^\top v/\lambda_\ast= \frac{\lambda_Dv^\top C_0v}{\lambda_\ast(\lambda_D+\lambda)^3}\geq \frac{\lambda_D^2}{\lambda_\ast(\lambda_D+\lambda_u)^3\|A^\top \Gamma^{-1}A\|}.
\]
So if we take $c$ as
\[
c=\frac{2\lambda_D^2}{\lambda_\ast(\lambda_D+\lambda_u)^3\|A^\top \Gamma^{-1}A\|},
\]
\eqref{eqn:A2} is verified.

Next, we note that by Taylor's theorem, there are some $w_k,w'_k$ between $\lambda_k-h_k$ and $\lambda_k+h_k$ such that 
\begin{align*}
|\widetilde\partial_\lambda f(\lambda_k,Z)-\partial_\lambda f(\lambda_k,Z)|
=\frac16h^2_k |\partial^3_{\lambda^3} f(w_k, Z)+\partial^3_{\lambda^3} f(w'_k, Z)|.
\end{align*}
Therefore 
\begin{align*}
\E |\widetilde\partial_\lambda f(\lambda_k,Z)-\partial_\lambda f(\lambda_k,Z)|^2
=\frac1{18}h^4_k (\E |\partial^3_{\lambda^3} f(w_k, Z)|^2+\E |\partial^3_{\lambda^3} f(w'_k, Z)|^2).
\end{align*}
We will show that there is dimension free constant $B_\lambda$ that may depend on $\lambda$ such that 
\begin{equation}
\label{tmp:p3}
\E |\partial^3_{\lambda^3} f(w_k, Z)|^2\leq B_\lambda,\quad \text{and} \quad \E |\partial^3_{\lambda^3} f(w'_k, Z)|^2\leq B_\lambda.
\end{equation}
This comes from the fact that each component of $\partial^3_{\lambda^3} f(w_k, Z)$ can be written as $\Tr(\Sigma C_v)$ or $\Tr(\Psi B)$ or $\Tr(\Sigma C_\xi)$, with some $\Sigma$ and $\Psi$. Here we define
\[
C_v= vv^\top,\quad B=\xi v^\top,\quad C_\xi=\xi\xi^\top.
\]
Then Lemma \ref{lem:chaining} with $n=1$ shows that for some universal constant $C$
\[
\E (\Sigma C_v)\leq 2|\Tr(\Sigma)|^2+C(\|\Sigma\|_F^3+\|\Sigma\|^3). 
\]
\[
\E (\Sigma C_\xi)\leq 2|\Tr(\Sigma)|^2+C(\|\Sigma\|_F^3+\|\Sigma\|^3). 
\]
\[
\E (\Psi B)\leq C(\|\Psi\|_F^3+\|\Psi\|^3). 
\]
Meanwhile, the $\Sigma$ matrices in $\partial^3_{\lambda^3} f(w_k, Z)$ are of form $P_j$ or  $D^\top P_j D$, which we know have bounded operator, trace and Frobenius norms, from the proof of Proposition \ref{prop:gensgd}, and the $\Psi$ matrix is of form $P_j D$, so $\|P_j D\|_F\leq \|P_j\|_F\|D\|, \|P_j D\|\leq \|P_j\|\|D\|$. So we can conclude that there is a dimension free constant $B$, such that \eqref{tmp:p3} holds.

Therefore  \eqref{eqn:A3} is verified by Jensen's inequality
\begin{align*}
|\E\widetilde\partial_\lambda f(\lambda_k,Z)-\partial_\lambda F(\lambda_k)|^2&=
|\E\widetilde\partial_\lambda f(\lambda_k,Z)-\E\partial_\lambda f(\lambda_k,Z)|^2\\
&\leq \E|\widetilde\partial_\lambda f(\lambda_k,Z)-\partial_\lambda f(\lambda_k,Z)|^2\leq \frac{1}{9}h_k^4B_\lambda. 
\end{align*}
Finally, because $\lambda \in [\lambda_l,\lambda_u]$, so $B_\lambda$ can be bounded as well. 

To prove that \eqref{eqn:A4} is satisfied, note that by Young's inequality 
\[
\E|\widetilde\partial_\lambda f(\lambda_k,Z)|^2\leq 2\E|\widetilde\partial_\lambda f(\lambda_k,Z)-\partial_\lambda f(\lambda_k,Z)|^2+2\E |\partial_\lambda f(\lambda_k,Z)|^2.
\]
Since we have already bounded $\E|\widetilde\partial_\lambda f(\lambda_k,Z)-\partial_\lambda f(\lambda_k,Z)|^2$, it suffices to bound $\E |\partial_\lambda f(\lambda_k,Z)|^2$ by a dimension independent constant $A_\lambda$. But each component of $\partial_{\lambda} f(w_k, Z)$ can be written as $\Tr(\Sigma C_v)$ or $\Tr(\Psi B)$ or $\Tr(\Sigma C_\xi)$, with some $\Sigma$ and $\Psi$. Then Lemma \ref{lem:chaining} with $n=1$ shows that for some universal constant $C$
\[
\E (\Sigma C_v)\leq 2|\Tr(\Sigma)|^2+C(\|\Sigma\|_F^3+\|\Sigma\|^3). 
\]
\[
\E (\Sigma C_\xi)\leq 2|\Tr(\Sigma)|^2+C(\|\Sigma\|_F^3+\|\Sigma\|^3). 
\]
\[
\E (\Psi B)\leq C(\|\Psi\|_F^3+\|\Psi\|^3). 
\]
Meanwhile, the $\Sigma$ matrices in $\partial^3_{\lambda^3} f(w_k, Z)$ are of form $P_j$ or  $D^\top P_j D$, which we know have bounded operator, trace and Frobenius norms. And the $\Psi$ matrix is of form $P_j D$, so $\|P_j D\|_F\leq \|P_j\|_F\|D\|, \|P_j D\|\leq \|P_j\|\|D\|$. So we can conclude that there is a dimension free constant $A_\lambda$, such that $\E |\partial_\lambda f(\lambda_k,Z)|^2\leq A_\lambda$. Finally, we note that $\lambda\in [\lambda_l,\lambda_u]$, we have our conclusion.

\end{proof}

\end{document}